\documentclass[11pt]{article}

\usepackage[a4paper,margin=25mm]{geometry}
\usepackage[T1]{fontenc}
\usepackage{lmodern}
\usepackage{microtype}
\usepackage{amsmath,amssymb,amsthm,mathtools,bbm}
\usepackage{graphicx}
\usepackage{aliascnt}
\usepackage{enumitem}
\usepackage[hidelinks]{hyperref}
\usepackage[nameinlink,capitalise]{cleveref}
\usepackage{xcolor}

\hypersetup{
  pdftitle={Upper-shadow comparisons on the slice and the Frankl--Tokushige product conjectures},
  pdfauthor={Fan Chang, Hong Liu, Miao Liu},
  pdfsubject={Discrete isoperimetry and extremal set theory},
  pdfkeywords={upper shadows, discrete isoperimetry, Boolean functions, cross-intersecting families, biased measure}
}

\newtheorem{theorem}{Theorem}[section]
\newaliascnt{lemma}{theorem}
\newtheorem{lemma}[lemma]{Lemma}
\aliascntresetthe{lemma}
\newaliascnt{corollary}{theorem}
\newtheorem{corollary}[corollary]{Corollary}
\aliascntresetthe{corollary}
\newaliascnt{proposition}{theorem}
\newtheorem{proposition}[proposition]{Proposition}
\aliascntresetthe{proposition}

\theoremstyle{remark}
\newaliascnt{remark}{theorem}
\newtheorem{remark}[remark]{Remark}
\aliascntresetthe{remark}
\newtheorem{claim}{Claim}

\newcommand{\F}{\mathcal F}
\newcommand{\A}{\mathcal A}
\newcommand{\Sstar}{\mathcal S}
\newcommand{\e}{\mathrm e}

\title{Upper-shadow comparisons on the slice and the Frankl--Tokushige product conjectures}
\author{Fan Chang\thanks{School of Statistics and Data Science, Nankai University, Tianjin, China; and Extremal Combinatorics and Probability Group, Institute for Basic Science, Daejeon, South Korea. Email: \texttt{1120230060@mail.nankai.edu.cn}. Supported by the National Natural Science Foundation of China under grant 124B2019 and by the Institute for Basic Science under grant IBS-R029-C4.}
\and Hong Liu\thanks{Extremal Combinatorics and Probability Group, Institute for Basic Science, Daejeon, South Korea. Email: \texttt{hongliu@ibs.re.kr}. Supported by the Institute for Basic Science under grant IBS-R029-C4.}
\and Miao Liu\thanks{Research Center for Mathematics and Interdisciplinary Sciences, Shandong University, Qingdao, China; and Extremal Combinatorics and Probability Group, Institute for Basic Science, Daejeon, South Korea. Email: \texttt{liumiao10300403@163.com}. Supported by the China Scholarship Council, the National Natural Science Foundation of China under Grant Nos. 12571352 and 12231014, and the Institute for Basic Science under grant IBS-R029-C4.}}
\date{}

\begin{document}
\maketitle

\begin{abstract}
Let $\mathcal{F}\subseteq\binom{[n]}k$, and let $\partial_{k\to\ell}\mathcal{F}$ be the family of all $\ell$-sets containing a member of $\mathcal{F}$. Writing $\mu_k(\mathcal{F})=|\mathcal{F}|/\binom nk$, $p=k/n$, and $q=\ell/n$,
we prove
$$
\mu_\ell(\partial_{k\to\ell}\F)
\geq
\begin{cases}
\displaystyle
\mu_k(\mathcal{F})^{\frac{\log q}{\log p}},
&0\leq \mu_k(\mathcal{F})\leq p^2,\\[2mm]
\displaystyle
q\left[1-\left(1-\frac{\mu_k(\mathcal{F})}{p}\right)^{
\frac{\log(1-q)}{\log(1-p)}}\right],
&p^2\leq \mu_k(\mathcal{F})\leq p,\\[3mm]
\displaystyle
1-(1-\mu_k(\mathcal{F}))^{\frac{\log(1-q)}{\log(1-p)}},
&p\leq \mu_k(\mathcal{F})\leq1.
\end{cases}
$$
This yields a dimension-free closed-form lower bound for the finite Kruskal--Katona profile and each branch is asymptotically sharp.

As the main application, we settle both the uniform and biased product
conjectures of Frankl and Tokushige. If $r\geq2$,
$0\leq k_i\leq(r-1)n/r$, and
$\mathcal{F}_i\subseteq\binom{[n]}{k_i}$ are $r$-cross-intersecting,
then
$$
\prod_{i=1}^r\mu_{k_i}(\mathcal{F}_i)\leq\prod_{i=1}^r\frac{k_i}{n}.
$$
In biased product setting, we prove that for $r$-cross-intersecting families $\mathcal{A}_i\subseteq 2^{[n]}$ and $0\leq p_i\leq(r-1)/r$, $\prod_{i=1}^r\mu_{p_i}(\A_i)\leq\prod_{i=1}^rp_i.$ We further determine all equality cases in both settings.

\medskip
\noindent\textbf{2020 Mathematics Subject Classification.}
Primary 05D05; Secondary 06A07, 06E30.

\medskip
\noindent\textbf{Keywords.}
Kruskal--Katona theorem; isoperimetry; cross-intersecting families; biased measure; Boolean functions.
\end{abstract}

\section{Introduction}

Discrete isoperimetric problems ask how small the boundary of a set can be when its size is prescribed. If $G=(V,E)$ is a finite graph and $X\subseteq V$, its \emph{external vertex boundary} is
$$
\partial_GX=\{v\in V\setminus X:uv\in E\text{ for some }u\in X\}.
$$
The vertex-isoperimetric problem is to minimize $|\partial_GX|$ among sets $X$ of a given size; see~\cite{Harper2004} for background.

Throughout, $n$ is a positive integer, $[n]=\{1,\ldots,n\}$, and
$\binom{[n]}k$ denotes the $k$-th level of the Boolean lattice
$2^{[n]}$. For $0\leq k\leq\ell\leq n$, let $\Gamma_n(k,\ell)$ be the
bipartite graph with vertex classes $\binom{[n]}k$ and $\binom{[n]}\ell$, where $A$ is adjacent to $B$ precisely when $A\subseteq B$; its edges are directed from level $k$ to level $\ell$. For $\F\subseteq\binom{[n]}k$, the resulting directed vertex boundary is the \emph{upper $\ell$-shadow}
$$
\partial_{k\to\ell}\F
=N_{\Gamma_n(k,\ell)}(\F)
=\left\{B\in\binom{[n]}\ell:A\subseteq B
\text{ for some }A\in\F\right\}.
$$
When $k=\ell$, the graph is a perfect matching and $\partial_{k\to k}\F=\F$. Thus the upper-shadow problem is the
directed vertex-isoperimetric problem in $\Gamma_n(k,\ell)$.

For distinct $A,B\in\binom{[n]}k$, write $A<_{\rm lex}B$ if the least
element of $A\mathbin{\triangle}B$ belongs to $A$. A family $\mathcal{L}\subseteq\binom{[n]}k$ is a \emph{lexicographic segment} if
it is an initial segment of this order. For $0\leq m\leq\binom nk$, let $\mathcal L_{n,k}(m)$ denote the unique
lexicographic segment of cardinality $m$. We use the following upper-shadow form of the Kruskal--Katona theorem.

\begin{theorem}[Kruskal~\cite{K1963}--Katona~\cite{K1968}]
Let $0\leq k\leq\ell\leq n$ and
$\F\subseteq\binom{[n]}k$. Then
$$
\bigl|\partial_{k\to\ell}\F\bigr|
\geq\bigl|\partial_{k\to\ell}
\mathcal L_{n,k}(|\F|)\bigr|.
$$
\end{theorem}

The theorem determines the directed vertex-isoperimetric profile of $\Gamma_n(k,\ell)$: a lexicographic segment on the $k$-th level has the smallest possible boundary on the $\ell$-th level. By taking complements and reversing the order of the ground set, this is equivalent to the usual lower-shadow theorem, where the extremal family is an initial segment of the \emph{colexicographic order}, defined by $A<_{\rm colex}B$ when the greatest element of $A\mathbin{\triangle}B$ belongs to $B$. This fundamental result of Kruskal and Katona~\cite{K1968,K1963} characterizes the possible face numbers of simplicial complexes and is used throughout extremal set theory; see~\cite{Bollobas1986,FT2018book}. Its analytic and robust forms also connect expansion on the slice with influences and the analysis of monotone Boolean functions; see, for example,~\cite{FOW2019,OW2009}.

Our first result gives an upper-shadow bound in terms of the density of the family and the ratios $k/n$ and $\ell/n$.

For $\F\subseteq\binom{[n]}k$, write $\mu_k(\F):=|\F|/\binom nk$ for its density on the $k$-th slice. For $0<p\leq q<1$, define $\Theta_{p,q}:[0,1]\to[0,1]$ by
\begin{equation}\label{eq:theta-profile}
\Theta_{p,q}(x)=
\begin{cases}
\displaystyle
x^{\frac{\log q}{\log p}},
&0\leq x\leq p^2,\\[2mm]
\displaystyle
q\left[1-\left(1-\frac{x}{p}\right)^{
\frac{\log(1-q)}{\log(1-p)}}\right],
&p^2\leq x\leq p,\\[3mm]
\displaystyle
1-(1-x)^{\frac{\log(1-q)}{\log(1-p)}},
&p\leq x\leq1.
\end{cases}
\end{equation}
The three formulas agree at $x=p^2$ and $x=p$; hence $\Theta_{p,q}$ is continuous and increasing on $[0,1]$. Here and
throughout, $\log$ denotes the natural logarithm. To include the
endpoint cases, set
\begin{equation}\label{eq:theta-profile-endpoints}
\begin{aligned}
&\Theta_{0,q}(x):=\mathbbm{1}_{\{x>0\}}
&& (0<q\leq1),\\
&\Theta_{p,1}(x):=\mathbbm{1}_{\{x>0\}}
&& (0<p<1),\\
&\Theta_{0,0}(x)=\Theta_{1,1}(x):=x.
\end{aligned}
\end{equation}
The first line is the pointwise limit of \eqref{eq:theta-profile} as
$p\downarrow0$ for fixed $0<q<1$, and the second is the corresponding
limit as $q\uparrow1$ for fixed $0<p<1$. These conventions also give
the correct upper-shadow bound when $k=0$ or $\ell=n$. In particular,
$\Theta_{p,p}(x)=x$ for all $p,x\in[0,1]$.

With this notation, our main upper-shadow estimate is the following.

\begin{theorem}\label{thm:shadow-main}
Let $1\leq k\leq\ell< n$, put $p=k/n$ and $q=\ell/n$, and let
$\F\subseteq\binom{[n]}k$. Then
\begin{equation}\label{ineq:upper-shadow-theta}
\mu_\ell(\partial_{k\to\ell}\F)
\geq\Theta_{p,q}(\mu_k(\F)).
\end{equation}
\end{theorem}
The bound is asymptotically sharp on each branch. For each fixed
integer $t\geq2$, the family
$\{A\in\binom{[n]}k:[t]\subseteq A\}$ asymptotically attains the first
branch. For each fixed integer $t\geq1$, the families
$\{A\in\binom{[n]}k:1\in A,\ A\cap\{2,\ldots,t+1\}\neq\varnothing\}$
and $\{A\in\binom{[n]}k:A\cap[t]\neq\varnothing\}$ asymptotically
attain the middle and third branches, respectively; see
\cref{prop:sharpness-families}. The Kruskal--Katona theorem determines
the exact finite profile, whereas \cref{thm:shadow-main} gives a  
closed-form lower envelope depending only on $p$, $q$, and the density;
see~\cref{fig:kk-theta}.

\begin{remark}\label{rem:role-of-shadow-bound}
The main motivation for~\cref{thm:shadow-main} was the
Frankl--Tokushige product conjectures; see~\cref{sec:uniform-proof}. Although the Kruskal--Katona theorem gives the exact finite shadow profile, its binomial-expansion form is not well suited to the unequal-level coupling arising here. Even after the critical-sum reduction, using the exact profile leads to a discrete optimization
depending on the finite parameters and on the binomial representation
of the family size. By contrast, \cref{thm:shadow-main} reduces the product problem to the analytic inequality in~\cref{thm:product-inequality}. The middle branch is essential: it gives
the extra control near the $1$-star density that is not supplied by
the power bound alone; see~\cref{rem:branch-roles}. The resulting estimate is also dimension-free: after writing $p=k/n$ and
$q=\ell/n$, the right-hand side of~\eqref{ineq:upper-shadow-theta} depends only on $p$, $q$, and $\mu_k(\F)$. Finally, our approach proceeds directly on the slice,
whereas Ellis, Keller and Lifshitz~\cite{EKL2019JEMS} first obtain isoperimetric stability results in the $p$-biased cube and then
transfer them to uniform families. Here the biased comparison in~\cref{thm:biased-shadow-comparison} is instead recovered afterwards by
passing to a product-measure limit.
\end{remark}

The functions also satisfy, for $0<p\leq q\leq r<1$,
\begin{equation}\label{eq:theta-composition}
\Theta_{q,r}\bigl(\Theta_{p,q}(x)\bigr)=\Theta_{p,r}(x).
\end{equation}
Since
$\partial_{\ell\to h}(\partial_{k\to\ell}\F)=\partial_{k\to h}\F$,
successive applications of \eqref{ineq:upper-shadow-theta} for
$(k,\ell)$ and $(\ell,h)$ recover the estimate for $(k,h)$. The
identity \eqref{eq:theta-composition} is proved in
\cref{sec:profile-geometry}.

\begin{figure}[t]
\centering
\includegraphics[width=.86\textwidth]{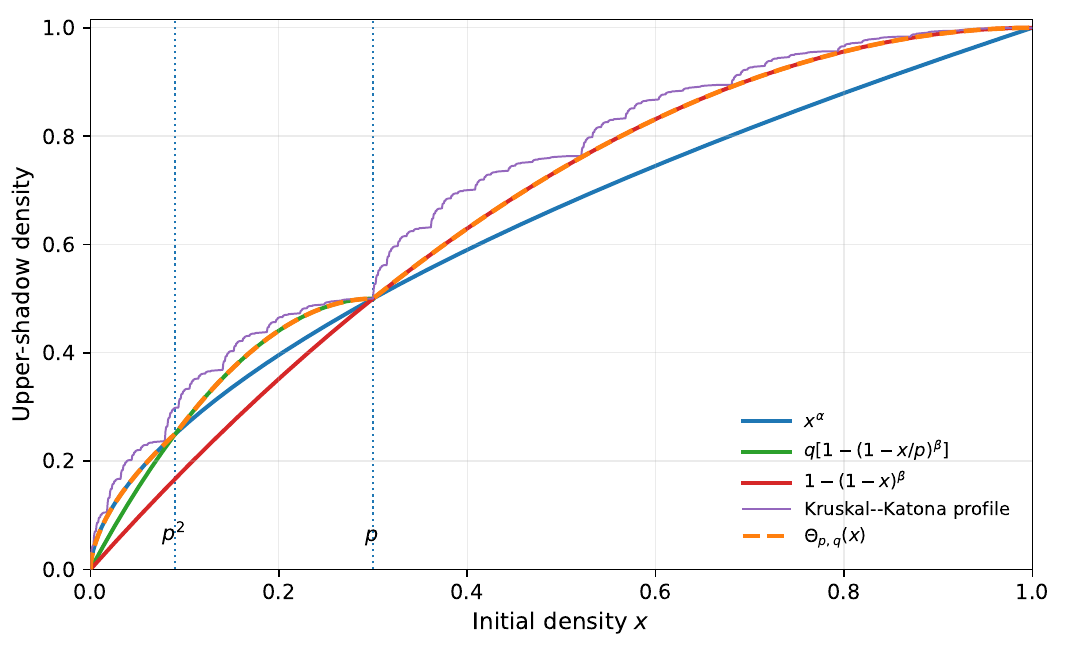}
\caption{The three branch functions and the resulting profile
$\Theta_{p,q}$, together with the exact Kruskal--Katona profile, for
$(n,k,\ell)=(20,6,10)$, where $p=3/10$ and $q=1/2$. The thin jagged
curve joins the exact values at $x=m/\binom{n}{k}$, while the thick
dashed curve is $\Theta_{p,q}$; the dotted vertical lines mark the
breakpoints $p^2$ and $p$.}
\label{fig:kk-theta}
\end{figure}

We shall use the following consequence of \cref{thm:shadow-main} in
both applications.

\begin{corollary}\label{cor:shadow-at-star}
Under the assumptions of \cref{thm:shadow-main},
\begin{equation}\label{eq:uniform-power-comparison}
\mu_\ell(\partial_{k\to\ell}\F)
\geq \mu_k(\F)^{\frac{\log q}{\log p}}.
\end{equation}
If $k<\ell$ and $\mu_k(\F)=p$, then equality holds in
\eqref{eq:uniform-power-comparison} if and only if
$\F=\{A\in\binom{[n]}k:i\in A\}$ for some $i\in[n]$.
\end{corollary}

\begin{remark}
Inequality~\eqref{eq:uniform-power-comparison} is sufficient to recover
both the Erd\H{o}s--Ko--Rado theorem~\cite{EKR1961} and the equal-level product theorem of Frankl and Tokushige~\cite{FT2011} for $r$-cross-intersecting $k$-uniform families. For the former, take $\ell=n-k$ and compare the upper shadow with the family of complements; see~\cref{cor:ekr-from-shadow}. For the latter, combine~\eqref{eq:uniform-power-comparison},~\cref{lem:uniform-critical-sum}, and AM--GM to obtain
$$
\prod_{i=1}^r\mu_k(\F_i)
\leq \left(\frac{k}{n}\right)^r
$$
whenever $\F_1,\ldots,\F_r\subseteq\binom{[n]}k$ are $r$-cross-intersecting and $rk\leq(r-1)n$; see~\cref{thm:equal-level-theorem}.

The exponent in \eqref{eq:uniform-power-comparison} also has the
following geometric interpretation. For $0<x\leq1$, the ratio $\log x/\log p$ is the unique real number $d$ such that $x=p^d$. For fixed $t$, the $t$-star $\{A\in\binom{[n]}k:[t]\subseteq A\}$ has density $\binom{n-t}{k-t}/\binom nk\to p^t$ as $n\to\infty$ with $k/n\to p$. Thus $d=t$ at the limiting density $p^t$. We call $\log x/\log p$ the \emph{effective codimension} of the density $x$ at $p$. If $\mu_k(\F)=p^d$, then~\eqref{eq:uniform-power-comparison} gives $\mu_\ell(\partial_{k\to\ell}\F)\geq q^d$; equivalently, for $\F\neq\varnothing$,
$$
\frac{\log\mu_\ell(\partial_{k\to\ell}\F)}{\log q}\leq\frac{\log\mu_k(\F)}{\log p}.
$$
Hence taking an upper shadow cannot increase effective codimension.
The same logarithmic normalization appears in the biased-measure
comparison of Ellis, Keller and Lifshitz~\cite[Lemma~2.6]{EKL2019JEMS}.
\end{remark}

For $0\leq p\leq1$, let $\mu_p$ be the product measure under which each element of $[n]$ is included independently with probability
$p$. For $\A\subseteq2^{[n]}$, write
$$
\mu_p(\A)=\sum_{A\in\A}p^{|A|}(1-p)^{n-|A|}.
$$
We refer to $(2^{[n]},\mu_p)$ as the $p$-biased cube. Taking the product-measure limit in~\cref{thm:shadow-main} gives the following theorem.

\begin{theorem}\label{thm:biased-shadow-comparison}
Let $\A\subseteq2^{[n]}$ be increasing and let $0\leq p\leq q\leq1$. Then
\begin{equation}\label{ineq:biased-theta}
\mu_q(\A)\geq\Theta_{p,q}(\mu_p(\A)).
\end{equation}
If $0<p<q<1$, then
\begin{equation}\label{eq:biased-power-comparison}
\mu_q(\A)\geq
\mu_p(\A)^{\frac{\log q}{\log p}}.
\end{equation}
For nonempty $\A$, equality in \eqref{eq:biased-power-comparison}
holds if and only if
$\A=\{A\subseteq[n]:S\subseteq A\}$ for some $S\subseteq[n]$.
\end{theorem}

The three branches of~\eqref{ineq:biased-theta} are attained by the following increasing families. For $2\leq t\leq n$, the family $\{A\subseteq[n]:[t]\subseteq A\}$ attains equality in the first branch. For $1\leq t\leq n-1$, the
family $\{A\subseteq[n]:1\in A,\ A\cap\{2,\ldots,t+1\}\neq\varnothing\}$
attains equality in the middle branch, while for $1\leq t\leq n$ the
family $\{A\subseteq[n]:A\cap[t]\neq\varnothing\}$ attains equality in the third branch.

\begin{remark}
Inequality~\eqref{eq:biased-power-comparison} recovers~\cite[Lemma~2.6]{EKL2019JEMS} and~\cite[Lemma~3.2]{CLL2026}. The third branch of \eqref{ineq:biased-theta} follows from
\cite[Lemma~2.7(ii)]{EKL2019JEMS} by contraposition and continuity, and
the middle branch follows in the same way from~\cite[Lemma~3.9]{EKL2019JEMS} with $t=1$. Ellis, Keller and Lifshitz
prove \eqref{eq:biased-power-comparison} by differentiating
$\log_p\mu_p(\A)$ and applying the Margulis--Russo formula~\cite{Margulis1974,Russo1982} together with the biased
edge-isoperimetric inequality. Their change-of-bias estimates are a
key step in the stability results of~\cite{EKL2019JEMS}.
\end{remark}

\subsection{The Frankl--Tokushige conjectures}

The Erd\H{o}s--Ko--Rado theorem~\cite{EKR1961} is the basic example of an intersection condition with a star extremizer: if $n\geq2k$
and $\F\subseteq\binom{[n]}k$ is intersecting, then $|\F|\leq\binom{n-1}{k-1}$, with equality attained by the family of all $k$-sets containing a fixed element. This principle has generated a broad theory of intersection problems; see, for example,~\cite{AK1997,BZ2023,DF1977,EFF2012,EFP11,EKL2023,FF1991,FT2016,FT2018book,F2008,KLMS2024,KZ2022,Wilson84}.

For several cross-intersecting families, the sizes can trade off
against one another, so no bound on an individual family captures the
extremal problem. The product of the normalized sizes treats the
families symmetrically, allows different uniformities or biases, and
agrees with the corresponding product for a common star. Frankl and
Tokushige asked whether the common-star construction remains extremal
whenever every choice of one set from each family has nonempty common
intersection.

Let $r\geq2$. Families $\F_1\subseteq\binom{[n]}{k_1},\ldots,\F_r\subseteq\binom{[n]}{k_r}$ are called $r$-cross-intersecting if $A_1\cap\cdots\cap A_r\neq\varnothing$ whenever $A_i\in\F_i$ for every $i\in[r]$. For $\F\subseteq\binom{[n]}k$, recall $\mu_k(\F)=|\F|/{\binom nk}$. If the families $\F_i$ are the corresponding levels of the same
$1$-star, then $\mu_{k_i}(\F_i)=\frac{k_i}{n}$ for every $i\in[r]$, and hence
$$\prod_{i=1}^r\mu_{k_i}(\F_i)=
\prod_{i=1}^r\frac{k_i}{n}.$$
Frankl and Tokushige~\cite[Conjecture~12.12]{FT2018book} conjectured
that this construction always maximizes the product whenever
$k_i\leq(r-1)n/r$. The equal-level case was proved in~\cite{FT2011}. The two-family case and several weighted or restricted cases were obtained in~\cite{Borg2016,MT1989,STT2017}; related sharp sum problems were
studied in~\cite{GMPS2023}. The unequal-level two-family theorem of Matsumoto and Tokushige was also used by Gromov in his homological isoperimetry for the torus, through the exterior-algebra description of its cohomology~\cite[p.~508]{Gromov2010}.

In our previous paper~\cite{CLL2026}, we studied the corresponding
problem for $p$-biased measures. The proof combined a coupling argument,
\eqref{eq:biased-power-comparison}, and the AM--GM inequality.

For unequal levels or unequal biases, there is no distinguished common
choice of the new levels. The coupling gives an additive inequality for
every vector $(q_1,\ldots,q_r)$ with $\sum_iq_i=r-1$, but does not
select the vector that yields the sharp product bound. Here we use a
similar coupling, replacing \eqref{eq:uniform-power-comparison} by the
stronger comparison in \cref{thm:shadow-main}; this settles both
conjectures.

\begin{theorem}\label{thm:uniform-main}
Let $n\geq1$ and $r\geq2$, let
$0\leq k_i\leq(r-1)n/r$ for every $i\in[r]$, and let $\F_i\subseteq\binom{[n]}{k_i}$ be $r$-cross-intersecting. Then
\begin{equation}\label{ineq:uniform-main}
\prod_{i=1}^r\mu_{k_i}(\F_i)\leq
\prod_{i=1}^r\frac{k_i}{n}.
\end{equation}
The equality cases are as follows.
\begin{itemize}
    \item If $k_i=0$ for some $i\in[r]$, equality holds for all such families.
    \item Suppose that $k_i>0$ for every $i\in[r]$, and either
    $r\geq3$, or $r=2$ and $k_1+k_2<n$. Then equality holds if and
    only if the families are full $1$-stars with a common center.
    \item If $r=2$ and $k_1+k_2=n$, then $k_1=k_2=\frac n2$. Equality holds if and only if there exists $\F\subseteq\binom{[n]}{n/2}$ with $|\F|=\frac12\binom{n}{n/2}$ such that
    $$\F_1=\F \quad\text{and}\quad \F_2=\left\{B\in\binom{[n]}{n/2}:[n]\setminus B\notin\F\right\}.$$
\end{itemize}
\end{theorem}

The parameter range in \cref{thm:uniform-main} is best possible. If $k>(r-1)n/r$, then any $r$ members of $\binom{[n]}k$ have nonempty common intersection, since
$|A_1\cap\cdots\cap A_r|\geq n-\sum_{i=1}^r|[n]\setminus A_i|=rk-(r-1)n>0$. Taking all $r$ families equal to $\binom{[n]}k$ then gives product $1>(k/n)^r$.

The same coupling and~\cref{thm:biased-shadow-comparison} also settle the unequal-bias conjecture of Frankl and
Tokushige~\cite[Conjecture~12.13]{FT2018book}.

\begin{theorem}\label{thm:biased-main}
Let $n\geq1$ and $r\geq2$, let
$0\leq p_i\leq(r-1)/r$ for every $i\in[r]$, and let $\A_1,\ldots,\A_r\subseteq2^{[n]}$ be $r$-cross-intersecting. Then
$$\prod_{i=1}^r\mu_{p_i}(\A_i)\leq\prod_{i=1}^rp_i.$$
The equality cases are as follows.
\begin{itemize}
    \item If some $p_i=0$, both sides are zero for all such families.
    \item Suppose that $p_i>0$ for every $i$. If $r\geq3$, or if
    $r=2$ and $(p_1,p_2)\neq(1/2,1/2)$, then equality holds if and
    only if the families are full $1$-stars with a common center.
    \item If $r=2$ and $p_1=p_2=1/2$, equality holds if and only if there is an increasing family $\A\subseteq2^{[n]}$ with $\mu_{1/2}(\A)=1/2$ such that $$\A_1=\A\quad\text{and}\quad\A_2=\left\{B\subseteq[n]:[n]\setminus B\notin\A\right\}.$$
\end{itemize}
\end{theorem}

The biased range is also best possible. If $p>(r-1)/r$, then $\A=\left\{A\subseteq[n]:|A|>\frac{r-1}{r}n\right\}$ is $r$-wise intersecting, while $\mu_p(\A)\to1$ as $n\to\infty$. Consequently, $\mu_p(\A)^r>p^r$ for all sufficiently large $n$.

\subsection{Outline of the proof}

\paragraph{Proof strategy for~\cref{thm:uniform-main}.}
The first ingredient is the critical-sum coupling. If $0\leq\ell_i\leq n$ and $\sum_i\ell_i=(r-1)n$, then $r$-cross-intersecting families
$\F_i\subseteq\binom{[n]}{\ell_i}$ satisfy
$$
\sum_{i=1}^r\mu_{\ell_i}(\F_i)\leq r-1.
$$
Indeed, let $P_1,\ldots,P_r$ be a uniformly random ordered partition of
$[n]$ with $|P_i|=n-\ell_i$, and put $X_i=[n]\setminus P_i$. Each $X_i$ has the required marginal distribution, whereas $\bigcap_iX_i=\varnothing$.

Since upper shadows preserve $r$-cross-intersection, the same argument gives
$$
\sum_{i=1}^r\mu_{\ell_i}(\partial_{k_i\to\ell_i}\F_i)\leq r-1
$$
whenever $k_i\leq\ell_i\leq n$ and
$\sum_i\ell_i=(r-1)n$. Apart from the preservation of $r$-cross-intersection under upper shadows, the product proof uses the intersection hypothesis only
through these linear inequalities.

Put $p_i=k_i/n$ and $a_i=\mu_{k_i}(\F_i)/p_i$, and relabel the positive parameters so that $a_1\geq\cdots\geq a_r$. Let $j$ be the first index for which $\sum_{i\leq j}(1-p_i)\geq1$, and put $q=\sum_{i<j}(1-p_i)$. Applying the critical-sum inequality at the levels
$$
k_1,\ldots,k_{j-1},qn,n,\ldots,n
$$
gives
$$
\sum_{i<j}p_i(a_i-1)
+\mu_{qn}(\partial_{k_j\to qn}\F_j)-q\leq0.
$$
Put $\alpha=\log q/\log p_j$ and $\beta=\log(1-q)/\log(1-p_j)$. By \cref{thm:shadow-main},
$$
\mu_{qn}(\partial_{k_j\to qn}\F_j)-q
\geq
\begin{cases}
q(a_j^\alpha-1),&0<a_j\leq p_j,\\[1mm]
-q(1-a_j)^\beta,&p_j\leq a_j\leq1,\\[1mm]
(1-p_j)^\beta-(1-p_ja_j)^\beta,
&1\leq a_j\leq \frac{1}{p_j}.
\end{cases}
$$
For $a_j\in(0,p_j]$ and $a_j\in[p_j,1)$, respectively, the first two bounds together with~\eqref{ineq:uniform-critical-sum-constraint} give hypotheses~(i) and~(ii) of \cref{thm:product-inequality}. For $a_j\in[1,1/p_j]$, the third bound, the same constraint, and the
ordering $a_1\geq\cdots\geq a_r$ give
$\prod_{i=1}^ra_i\leq1$, with equality only when $a_1=\cdots=a_r=1$.

It remains to classify equality. The preceding argument first gives
$\mu_{k_i}(\F_i)=k_i/n$ for every $i$. If $\sum_i(n-k_i)>n$, then~\eqref{eq:uniform-power-comparison}, \cref{lem:uniform-critical-sum}, and the equality statement in~\cref{cor:shadow-at-star} show that every $\F_i$ is a full $1$-star; \cref{lem:common-star-center} then shows that their centers coincide. If $\sum_i(n-k_i)=n$, the equality cases
follow from \cref{lem:uniform-partition-equality} for $r\geq3$ and
from the complement relation for $r=2$.

\paragraph{Proof strategy for \cref{thm:shadow-main}.}
The proof of \cref{thm:shadow-main} proceeds by induction on $n$.
By \cref{lem:balanced-restriction}, choose a coordinate $i$ such that
$\mu_k(\F_0(i))\leq\mu_{k-1}(\F_1(i))$. Set $\F'=\partial_{k\to\ell}\F$. Then
$\F'_0(i)=\partial_{k\to\ell}\F_0(i)$ and $\F'_1(i)\supseteq\partial_{k-1\to\ell-1}\F_1(i)$. The two inductive bounds for $\F_0(i)$ and $\F_1(i)$ are combined by~\eqref{ineq:slice-two-point}, completing the inductive step.

Coordinate induction through a two-point inequality is a standard local-to-global method in discrete analysis. In Bellman-function language, \eqref{ineq:slice-two-point} is the discrete concavity condition corresponding to the two density identities for $\F$ and $\F'$. Its proof treats the three intervals in
\eqref{eq:theta-profile} separately and uses~\cref{lem:uniform-exponent-comparison,lem:two-point-powers,lem:theta-comparison}.
This induction scheme appears in Talagrand's proof of edge-isoperimetric inequalities on the discrete cube~\cite{T1993} and
in subsequent work on sharp discrete isoperimetric and functional
inequalities~\cite{BIM2023,BIMP2025,BG1999,DIPV2026,DIR2024,DIRX2026}.

On the slice, the concavity inequality depends on four parameters. The
families $\F_0(i)$ and $\F_1(i)$ lie on levels $k$ and $k-1$ of
$[n]\setminus\{i\}$, while $\F'_0(i)$ and $\F'_1(i)$ lie on levels
$\ell$ and $\ell-1$. Hence the two inductive estimates involve the
parameter pairs $(p_0,q_0)$ and $(p_1,q_1)$, whereas the density
identities for $\F$ and $\F'$ use the weights $(1-p,p)$ and
$(1-q,q)$, respectively. The new analytic step is to establish
\eqref{ineq:slice-two-point} with these parameter pairs and weights;
the coordinate-induction principle itself is classical.

\medskip\noindent\emph{Organization.}
\Cref{sec:pre} fixes the notation and records the preliminary
estimates. \Cref{sec:shadow} proves the two-point inequality and the
upper-shadow theorem, derives the biased-measure and infinitesimal
forms, proves the composition identity, identifies asymptotically
sharp families, and deduces the Erd\H{o}s--Ko--Rado theorem from
\eqref{eq:uniform-power-comparison}. \Cref{sec:uniform-proof} proves
the Frankl--Tokushige conjectures and classifies equality. The
analytic product inequality is proved in
\cref{sec:product-inequality}, and \cref{sec:concluding} discusses the
$r$-cross $t$-intersecting problem.

\section{Preliminaries}\label{sec:pre}
This section collects the notation and elementary estimates used throughout the proof. 

For a finite set $X$ and an integer $k$, let $\binom{X}{k}$ denote the collection of all $k$-element subsets of $X$. We refer to $\binom{X}{k}$ as the $k$-th level, or the $k$-th slice, of the Boolean lattice $2^X$. For $j\in[n]$, let $\Sstar_j=\{A\subseteq[n]:j\in A\}$ denote the $1$-star centered at $j$. The families $\Sstar_j\cap\binom{[n]}{k_1},\ldots, \Sstar_j\cap\binom{[n]}{k_r}$ are referred to as the corresponding levels of a common $1$-star.

Let $\F\subseteq\binom{[n]}k$ and $i\in[n]$. We define the $0$-restriction of $\F$ at $i$ and the link of $\F$ at $i$, respectively, by
$$
\begin{aligned}
\F_0(i)&=\{A\in\F:i\notin A\}\subseteq\binom{[n]\setminus\{i\}}{k},\\
\F_1(i)&=\{A\setminus\{i\}:A\in\F,\ i\in A\}\subseteq
\binom{[n]\setminus\{i\}}{k-1}.
\end{aligned}
$$
Their normalized densities are
$$
\mu_k(\F_0(i))=\frac{|\F_0(i)|}{\binom{n-1}{k}},
\qquad
\mu_{k-1}(\F_1(i))=\frac{|\F_1(i)|}{\binom{n-1}{k-1}}.
$$
Thus, writing $p=k/n$, we have
$$
\mu_k(\F)=(1-p)\mu_k(\F_0(i))+p\mu_{k-1}(\F_1(i)).
$$
When the distinguished coordinate is $n$, we identify
$[n]\setminus\{n\}$ with $[n-1]$ and abbreviate
$\F_0(n),\F_1(n)$ to $\F_0,\F_1$, respectively.

Averaging over the coordinate gives the choice used in the induction.

\begin{lemma}\label{lem:balanced-restriction}
Let $1\leq k<n$ and $\F\subseteq\binom{[n]}k$. Then
$$
\frac{1}{n}\sum_{i=1}^n\mu_k(\F_0(i))=
\frac{1}{n}\sum_{i=1}^n\mu_{k-1}(\F_1(i))=\mu_k(\F).
$$
In particular, there exists a coordinate $i\in[n]$ such that $\mu_k(\F_0(i))\leq \mu_{k-1}(\F_1(i))$.
\end{lemma}

\begin{proof}
Each member of $\F$ avoids exactly $n-k$ coordinates and contains exactly $k$ coordinates. Double counting, together with
$
\binom{n-1}k=\frac{n-k}{n}\binom nk$ and $\binom{n-1}{k-1}=\frac{k}{n}\binom nk,
$
gives both identities. The final assertion follows because the two averages are equal.
\end{proof}

\subsection{Dual families}
For $\F\subseteq\binom{[n]}{k}$, define its \emph{uniform dual family} by
$$
\F^*:=\left\{B\in\binom{[n]}{n-k}:
[n]\setminus B\notin\F\right\}.
$$
The complement map gives $(\F^*)^*=\F$ and $\mu_{n-k}(\F^*)=1-\mu_k(\F)$. Similarly, for $\A\subseteq2^{[n]}$, define its \emph{dual family} by
$$\A^*:=\{B\subseteq[n]:[n]\setminus B\notin\A\}.$$
Then $(\A^*)^*=\A$ and $\mu_q(\A^*)=1-\mu_{1-q}(\A)$ for every $0\leq q\leq1$.

The next lemma records the covering property needed for uniform cross-intersecting families.
\begin{lemma}\label{lem:uniform-dual-family}
Let $r\geq2$, and suppose that
$\F_1\subseteq\binom{[n]}{k_1},\dots,\F_r\subseteq\binom{[n]}{k_r}$ are $r$-cross-intersecting. Then, for every $i\in[r]$,
$$\mu_{n-k_i}(\F_i^*)=1-\mu_{k_i}(\F_i). $$
Moreover, for every ordered partition $[n]=P_1\mathbin{\dot\cup}\cdots\mathbin{\dot\cup}P_r$ satisfying $|P_i|=n-k_i$ for every $i\in[r]$, there exists some $i\in[r]$ such that $P_i\in\F_i^*$.
\end{lemma}

\begin{proof}
The complement map $B\mapsto[n]\setminus B$ is a bijection from $\binom{[n]}{n-k_i}$ onto $\binom{[n]}{k_i}$, and hence
$$|\F_i^*|=\binom{n}{k_i}-|\F_i|.$$
Since
$\binom{n}{n-k_i}=\binom{n}{k_i}$, we obtain
$$
\mu_{n-k_i}(\F_i^*)=\frac{|\F_i^*|}{\binom{n}{n-k_i}}=1-\frac{|\F_i|}{\binom{n}{k_i}}=1-\mu_{k_i}(\F_i).
$$

It remains to prove the covering property. Let $[n]=P_1\mathbin{\dot\cup}\cdots\mathbin{\dot\cup}P_r$ be an ordered partition with $|P_i|=n-k_i$ for every $i\in[r]$.
Suppose, to the contrary, that
$P_i\notin\F_i^*$ for every $i\in[r]$.
By the definition of $\F_i^*$, this implies $[n]\setminus P_i\in\F_i$ for every $i\in[r]$. On the other hand, since $P_1,\ldots,P_r$ form a partition of $[n]$,
$$
\bigcap_{i=1}^r ([n]\setminus P_i)
=[n]\setminus\bigcup_{i=1}^rP_i
=\varnothing.
$$
This contradicts the assumption that
$\F_1,\ldots,\F_r$ are $r$-cross-intersecting. This completes the proof.
\end{proof}

For increasing families in $2^{[n]}$, the same complement argument requires no restriction on the sizes of the parts in the ordered partition. We use this version in the biased setting.

\begin{lemma}[{\cite[Lemma~4.1]{CLL2026}}]\label{lem:biased-dual-family}
    Let $r\geq2$. Suppose that $\A_1,\dots,\A_r\subseteq 2^{[n]}$ are increasing and $r$-cross-intersecting. Then every $\A_i^*$ is an increasing family, $\mu_{1/r}(\A_i^*)=1-\mu_{(r-1)/r}(\A_i)$ for every $i\in[r]$, and for every ordered partition $[n]=P_1\mathbin{\dot\cup}\cdots\mathbin{\dot\cup}P_r$, there exists some $i\in[r]$ such that $P_i\in\A_i^*$.
\end{lemma}

\subsection{Auxiliary inequalities}

The restriction decomposition reduces the upper-shadow comparison to
two one-dimensional estimates. The first compares the logarithmic
exponents arising after restriction, and the second combines the
corresponding powers.

\begin{lemma}\label{lem:uniform-exponent-comparison}
Let $0<a\leq b<1$.
\begin{enumerate}[label={\rm (\roman*)}]
\item If $0\leq\delta<a$, then
$\frac{\log\frac{a-\delta}{1-\delta}}{\log\frac{b-\delta}{1-\delta}}\geq\frac{\log a}{\log b}$.
\item If $0\leq\delta<1-b$, then
$\frac{\log\frac{a}{1-\delta}}{\log\frac{b}{1-\delta}}\geq
\frac{\log a}{\log b}$.
\end{enumerate}
\end{lemma}

\begin{proof}
For (ii), put $c=1-\delta$. The derivative of $\frac{\log a-\log c}{\log b-\log c}$ with respect to $\log c$ is $\frac{\log a-\log b}{(\log b-\log c)^2}\leq 0$. Since $\log c$ decreases with $\delta$, the ratio is nondecreasing in $\delta$.

For (i), put $u_x=(x-\delta)/(1-\delta)$. Direct differentiation gives
$$
\frac{d}{d\delta}\log(-\log u_x)=
\frac{1}{1-\delta}\frac{1-u_x}{u_x(-\log u_x)}.
$$
The last factor decreases with $u_x$: after writing $u_x=\e^{-v}$, it becomes $(\e^v-1)/v$, which increases with $v$. Since $u_a\leq u_b$, the ratio $(-\log u_a)/(-\log u_b)$ is nondecreasing in $\delta$, proving (i).
\end{proof}

\begin{lemma}\label{lem:two-point-powers}
Let $0<t<1$ and $0\leq z\leq1$.
\begin{enumerate}[label={\rm (\roman*)}]
\item If $\gamma\geq1$, then $t^\gamma+(1-t^\gamma)z^\gamma\leq\bigl(t+(1-t)z\bigr)^\gamma$.
\item If $0<\beta\leq1$, then $\bigl(t+(1-t)z\bigr)^\beta\leq t^\beta+(1-t^\beta)z^\beta$.
\end{enumerate}
\end{lemma}
\begin{proof}
We first prove part {\rm (i)}. The case $\gamma=1$ is an equality, so
let $\gamma>1$ and set
$$
F(z)=\bigl(t+(1-t)z\bigr)^\gamma-t^\gamma-(1-t^\gamma)z^\gamma.
$$
We have $F(0)=F(1)=0$, and for $z>0$, the sign of $F'(z)$ is the sign of
$$
(1-t)\left(\frac{t}{z}+1-t\right)^{\gamma-1}-(1-t^\gamma).
$$
This expression decreases strictly from $+\infty$ to $(1-t)-(1-t^\gamma)<0$. Hence $F$ first increases and then decreases, so $F\geq0$ on $[0,1]$.

For part {\rm (ii)}, apply part {\rm (i)} with exponent $\frac{1}{\beta}$, parameter $t^\beta$, and variable $z^\beta$, and then raise the resulting inequality to the power $\beta$.
\end{proof}

The following lemma compares the three functions in~\eqref{eq:theta-profile} and shows that $\Theta_{p,q}$ is their
pointwise maximum, with the middle function restricted to $[0,p]$.

\begin{lemma}\label{lem:theta-comparison}
Let $0<p\leq q<1$. Then
\begin{enumerate}
    \item[(i)] $\Theta_{p,q}(x)\geq x^{\frac{\log q}{\log p}}$ for all $0\leq x\leq 1$;
    \item[(ii)] $\Theta_{p,q}(x)\geq q\left[1-\left(1-\frac{x}{p}\right)^{\frac{\log(1-q)}{\log(1-p)}}\right]$ for all $0\leq x\leq p$;
    \item[(iii)] $\Theta_{p,q}(x)\geq 1-(1-x)^{\frac{\log(1-q)}{\log(1-p)}}$ for all $0\leq x\leq 1$.
\end{enumerate}
\end{lemma}

\begin{proof}
Set $\alpha:=\frac{\log q}{\log p}$ and $\beta:=\frac{\log(1-q)}{\log(1-p)}$. Since $0<p\leq q<1$, we have $0<\alpha\leq 1\leq \beta$, and, by the definitions of $\alpha$ and $\beta$, $p^\alpha=q$ and $(1-p)^\beta=1-q$.

If $p=q$, then $\alpha=\beta=1$ and all three assertions are immediate. Hence assume that $p<q$, so that $0<\alpha<1<\beta$, and define
$$
U(t):=1-(1-t)^\beta
$$
on $t\in [0,1]$.
\begin{claim}\label{clm:power-star-crossover}
    $U(t)\leq t^\alpha \text{ for } 0\leq t\leq p\text{ and } U(t)\geq t^\alpha \text{ for } p\leq t\leq1$.
\end{claim}
For $0<t<1$, define
$$
\Phi(t):=\frac{\log(1-t^\alpha)}{\log(1-t)}.
$$
Let $f(t)=-\log(1-t^\alpha)$ and $g(t)=-\log(1-t)$. Then
$$
\frac{f'(t)}{g'(t)}
=\frac{\alpha t^{\alpha-1}(1-t)}{1-t^\alpha}
=\alpha\frac{1-t}{t^{1-\alpha}-t}.
$$
Moreover,
$$
\frac{d}{dt}\left(\frac{1-t}{t^{1-\alpha}-t}\right)
=
\frac{1-t^{-\alpha}\bigl((1-\alpha)+\alpha t\bigr)}
     {(t^{1-\alpha}-t)^2}
\leq0,
$$
where the last inequality follows from the weighted AM--GM inequality
$t^\alpha\leq(1-\alpha)+\alpha t$. Thus $\frac{f'}{g'}$ is decreasing. Since $f(0)=g(0)=0$, the quotient $\frac{f}{g}=\Phi$ is also decreasing; indeed, it is a weighted average of $\frac{f'}{g'}$ on $[0,t]$ with weight $g'$. Finally,
$$
\Phi(p)=\frac{\log(1-p^\alpha)}{\log(1-p)}
=\frac{\log(1-q)}{\log(1-p)}=\beta.
$$
The monotonicity of $\Phi$ is equivalent to~\cref{clm:power-star-crossover}.

We use the claim to compare the three formulas. On $[0,p^2]$, the first branch of $\Theta_{p,q}$ is exactly $x^\alpha$. If $p^2\leq x\leq p$, then
$\frac{x}{p}\in[p,1]$, and~\cref{clm:power-star-crossover} gives
$$
\Theta_{p,q}(x)=qU\left(\frac{x}{p}\right) \geq q\left(\frac{x}{p}\right)^\alpha=x^\alpha.
$$
If $p\leq x\leq1$, the same relation gives
$\Theta_{p,q}(x)=U(x)\geq x^\alpha$. This proves {\rm (i)}.

For {\rm (ii)}, equality holds on $[p^2,p]$. If $0\leq x\leq p^2$,
then $\frac{x}{p}\leq p$, and \cref{clm:power-star-crossover} gives
$$
qU\left(\frac{x}{p}\right)\leq q\left(\frac{x}{p}\right)^\alpha=x^\alpha=\Theta_{p,q}(x).
$$
Finally, the third comparison is an equality on $[p,1]$. For $0\leq x\leq p$, \cref{clm:power-star-crossover} gives $U(x)\leq x^\alpha$, and part {\rm (i)} gives $\Theta_{p,q}(x)\geq x^\alpha$. Hence $\Theta_{p,q}(x)\geq U(x)$, proving {\rm (iii)}.
\end{proof}

\section{Upper-shadow comparison}\label{sec:shadow}

\subsection{The two-point inequality}\label{subsec:two-point}

We begin with the scalar inequality that combines the two inductive
estimates.

\begin{lemma}[Two-point comparison for adjacent slices]\label{lem:slice-two-point}
Let $1\leq k<\ell<n$, and define
$$p=\frac kn,\quad q=\frac\ell n,\quad
p_0=\frac{k}{n-1},\quad q_0=\frac{\ell}{n-1},\quad
p_1=\frac{k-1}{n-1},\quad q_1=\frac{\ell-1}{n-1}.$$
Then, for $0\leq x_0\leq x_1\leq1$,
\begin{equation}\label{ineq:slice-two-point}
(1-q)\Theta_{p_0,q_0}(x_0)+q\Theta_{p_1,q_1}(x_1)\geq
\Theta_{p,q}\bigl((1-p)x_0+px_1\bigr).
\end{equation}
Moreover, if $(1-p)x_0+px_1=p$, then equality in \eqref{ineq:slice-two-point} holds if and only if $x_0=0$ and $x_1=1$.
\end{lemma}

The parameter pairs $(p_0,q_0)$ and $(p_1,q_1)$ differ, and the two
levels use different averaging weights. The proof treats the three
formulas in \eqref{eq:theta-profile} separately.

\begin{proof}[Proof of \cref{lem:slice-two-point}]
The argument below remains valid when $k$ and $\ell$ are regarded
as real parameters satisfying $1<k<\ell<n-1$. Under the endpoint
convention \eqref{eq:theta-profile-endpoints}, both sides of
\eqref{ineq:slice-two-point} are continuous as $k\downarrow1$ or
$\ell\uparrow n-1$. Hence it suffices to prove the inequality for
$1<k<\ell<n-1$. Put
$$
\alpha=\frac{\log q}{\log p},
\qquad
\beta=\frac{\log(1-q)}{\log(1-p)},
$$
and define $\alpha_i=\frac{\log q_i}{\log p_i}$ and $\beta_i=\frac{\log(1-q_i)}{\log(1-p_i)}$ for $i\in\{0,1\}$. Applying
\cref{lem:uniform-exponent-comparison} to $p,q$ and to
$1-q,1-p$ gives
$$
\alpha_0,\alpha_1\leq\alpha,
\qquad
\beta_0,\beta_1\geq\beta.
$$
For $\alpha_0$ and $\alpha_1$, use parts~{\rm (ii)} and~{\rm (i)},
respectively, and for $\beta_0$ and $\beta_1$,
use parts~{\rm (i)} and~{\rm (ii)}, respectively.
In particular,
$$
1-q_i=(1-p_i)^{\beta_i}\leq(1-p_i)^\beta,
\quad i\in\{0,1\}.
$$
We shall also use
$$
p_1<p<p_0,
\quad
(1-q)q_0=q(1-q_1),
\quad
(1-p)p_0=p(1-p_1).
$$
Set $x=(1-p)x_0+px_1$.

\medskip
\noindent\textbf{Case 1: $0\leq x\leq p^2$.} By \cref{lem:theta-comparison}, $$
\Theta_{p_0,q_0}(x_0)\geq x_0^{\alpha_0}\geq x_0^\alpha,
\quad
\Theta_{p_1,q_1}(x_1)\geq x_1^{\alpha_1}\geq x_1^\alpha.$$
If $x_1=0$, then $x_0=x=0$ and the required power bound is immediate. If $x_1>0$, \cref{lem:two-point-powers} (ii), applied with
$t=p$, $z=x_0/x_1$, and exponent $\alpha$, gives
$$\begin{aligned}
(1-q)\Theta_{p_0,q_0}(x_0)+q\Theta_{p_1,q_1}(x_1)\geq(1-q)x_0^\alpha+qx_1^\alpha\geq\bigl((1-p)x_0+px_1\bigr)^\alpha
=x^\alpha.
\end{aligned}$$

\medskip
\noindent\textbf{Case 2: $p^2\leq x\leq p$.} Since $x_0\leq x\leq p<p_0$, we always have $x_0<p_0$.
We distinguish two positions of $x_1$ relative to $p_1$.

\smallskip
\noindent\emph{Subcase 2.1: $x_1\leq p_1$.}
By \cref{lem:theta-comparison},
$$\Theta_{p_0,q_0}(x_0)\geq q_0\left(1-\left(1-\frac{x_0}{p_0}\right)^{\beta_0}\right),
\quad
\Theta_{p_1,q_1}(x_1)\geq q_1\left(1-\left(1-\frac{x_1}{p_1}\right)^{\beta_1}\right).$$
Using the
identity $(1-q)q_0=q(1-q_1)$, $1-q_1\leq(1-p_1)^\beta$, \cref{lem:two-point-powers}(i), and $(1-p)p_0=p(1-p_1)$, we obtain
$$\begin{aligned}
&(1-q)\Theta_{p_0,q_0}(x_0)+q\Theta_{p_1,q_1}(x_1)\geq
q\left[1-(1-q_1)\left(1-\frac{x_0}{p_0}\right)^\beta-q_1\left(1-\frac{x_1}{p_1}\right)^\beta\right] \\
& \geq q\left[1-(1-p_1)^\beta \left(1-\frac{x_0}{p_0}\right)^\beta
-\left(1-(1-p_1)^\beta\right)\left(1-\frac{x_1}{p_1}\right)^\beta\right] \\
&\geq q\left[1-\left((1-p_1)\left(1-\frac{x_0}{p_0}\right)+p_1\left(1-\frac{x_1}{p_1}\right)\right)^\beta\right]=q\left[1-\left(1-\frac{x}{p}\right)^\beta\right]=\Theta_{p,q}(x).
\end{aligned}
$$
The second inequality follows from
$\bigl((1-p_1)^\beta-(1-q_1)\bigr)\left(\left(1-\frac{x_0}{p_0}\right)^\beta-\left(1-\frac{x_1}{p_1}\right)^\beta\right)\geq0$,
and the third is \cref{lem:two-point-powers}(i), applied with
$t=1-p_1$ and $z=\left(1-\frac{x_1}{p_1}\right)/\left(1-\frac{x_0}{p_0}\right)$.

\smallskip
\noindent\emph{Subcase 2.2: $x_1\geq p_1$.}
By \cref{lem:theta-comparison},
$$\Theta_{p_0,q_0}(x_0)\geq q_0\left(1-\left(1-\frac{x_0}{p_0}\right)^{\beta_0}\right),
\quad
\Theta_{p_1,q_1}(x_1)\geq 1-\left(1-x_1\right)^{\beta_1}.$$
Using the identity $(1-q)q_0=q(1-q_1)$, we obtain
\begin{equation}\label{eq:two-point-mixed-bound}
    (1-q)\Theta_{p_0,q_0}(x_0)+q\Theta_{p_1,q_1}(x_1)\geq q+q(1-q_1)\left(1-\left(1-\frac{x_0}{p_0}\right)^\beta-\left(\frac{1-x_1}{1-p_1}\right)^\beta\right).
\end{equation}
A direct calculation gives $1-\frac{x}{p}=(1-p_1)(\frac{1-x_1}{1-p_1}-\frac{x_0}{p_0}).$  Since $x\leq p$, we have $\frac{1-x_1}{1-p_1}+1-\frac{x_0}{p_0}\geq 1$.
For fixed $s=\frac{1-x_1}{1-p_1}+1-\frac{x_0}{p_0}\in[1,2]$, the convex function
$t\mapsto t^\beta+(s-t)^\beta$ attains its maximum on
$[s-1,1]$ at an endpoint. Then
\begin{equation}\label{eq:convex-endpoint-bound}
\left(1-\frac{x_0}{p_0}\right)^\beta+\left(\frac{1-x_1}{1-p_1}\right)^\beta\leq 1+\left(\frac{1-x_1}{1-p_1}-\frac{x_0}{p_0}\right)^\beta.
\end{equation}
Using $1-q_1\leq (1-p_1)^\beta$, \eqref{eq:two-point-mixed-bound}, and \eqref{eq:convex-endpoint-bound}, we have
$$(1-q)\Theta_{p_0,q_0}(x_0)+q\Theta_{p_1,q_1}(x_1)\geq q-q(1-p_1)^\beta\left(\frac{1-x_1}{1-p_1}-\frac{x_0}{p_0}\right)^\beta=q\left[1-\left(1-\frac{x}{p}\right)^\beta\right]=\Theta_{p,q}(x).$$

\medskip
\noindent\textbf{Case 3: $p\leq x\leq 1$.}
Since $x_1\geq x\geq p>p_1$, the third formula in \eqref{eq:theta-profile} applies to $\Theta_{p_1,q_1}(x_1)$. We distinguish the two possible positions of $x_0$.

\smallskip
\noindent\emph{Subcase 3.1: $x_0\geq p_0$.}
Applying \cref{lem:two-point-powers}(i) with
$t=1-p$ and $z=(1-x_1)/(1-x_0)$, we obtain
$$\begin{aligned}
&(1-q)\Theta_{p_0,q_0}(x_0)+q\Theta_{p_1,q_1}(x_1)\geq
1-(1-q)(1-x_0)^{\beta_0}-q(1-x_1)^{\beta_1}\\
&\geq1-(1-q)(1-x_0)^\beta-q(1-x_1)^\beta\geq 1-(1-x)^\beta=\Theta_{p,q}(x).
\end{aligned}$$

\smallskip
\noindent\emph{Subcase 3.2: $x_0\leq p_0$.}
Since $\beta\geq 1$,
$$\left(1-\frac{x_0}{p_0}\right)^\beta+\left(\frac{1-x_1}{1-p_1}\right)^\beta\leq \left(1-\frac{x_0}{p_0}+\frac{1-x_1}{1-p_1}\right)^\beta.$$
The same initial estimate in \eqref{eq:two-point-mixed-bound} gives
$$
\begin{aligned}
    &(1-q)\Theta_{p_0,q_0}(x_0)+q\Theta_{p_1,q_1}(x_1)\geq q+q(1-q_1)\left(1-\left(1-\frac{x_0}{p_0}+\frac{1-x_1}{1-p_1}\right)^\beta\right)\\
    &=1-(1-q)\left[(1-q_0)+q_0\left(1-\frac{x_0}{p_0}+\frac{1-x_1}{1-p_1}\right)^\beta\right]\\
    &\geq 1-(1-q)\left((1-p_0)+p_0\left(1-\frac{x_0}{p_0}+\frac{1-x_1}{1-p_1}\right)\right)^\beta=1-(1-x)^\beta=\Theta_{p,q}(x).
\end{aligned}$$
The first equality follows from $(1-q)q_0=q(1-q_1)$, the second inequality follows from $(1-q_0)\leq (1-p_0)^\beta$ and \cref{lem:two-point-powers}(i), and the second equality follows from $\frac{1-x}{1-p}=(1-p_0)+p_0\left(1-\frac{x_0}{p_0}+\frac{1-x_1}{1-p_1}\right)$.

It remains to verify the equality assertion. Suppose first that
$1<k<\ell<n-1$ and $x=p$. Since $x_0\leq x_1$, we have $x_0\leq p\leq x_1$; hence $x_0<p_0$ and $x_1>p_1$, so we are in Subcase~2.2. Since $1-\frac{x}{p}=(1-p_1)(\frac{1-x_1}{1-p_1}-\frac{x_0}{p_0})$, we have $\frac{1-x_1}{1-p_1}-\frac{x_0}{p_0}=0$. Moreover, if the left-hand side of \eqref{eq:two-point-mixed-bound} equals $q$, then $\left(1-\frac{x_0}{p_0}\right)^\beta+\left(\frac{1-x_1}{1-p_1}\right)^\beta=1$. Since $\beta>1$, this forces $\left(1-\frac{x_0}{p_0}\right)\in\{0,1\}$. If $1-\frac{x_0}{p_0}=0$, then $x_0=p_0>p=x$, contrary to $x_0\leq x$. Therefore $1-\frac{x_0}{p_0}=1$ and $\frac{1-x_1}{1-p_1}=0$, which is equivalent to $x_0=0$ and $x_1=1$.

If $k=1$, then $p_1=0$. Under $x=p$, the inequalities $x_0\leq x_1$ imply $x_1>0$. If $x_0>0$, the endpoint convention gives $\Theta_{p_1,q_1}(x_1)=1$, while $\Theta_{p_0,q_0}(x_0)>0$, so the left-hand side of \eqref{ineq:slice-two-point} is strictly larger than $q$. Hence $x_0=0$, and then $x=p$ gives $x_1=1$. Finally, suppose that $k>1$ and $\ell=n-1$. Then $q_0=1$. If $x_0>0$, the first term on the left-hand side is $1-q$, while $x_1\geq p>p_1$ and strict monotonicity of the third branch gives $\Theta_{p_1,q_1}(x_1)>q_1$. Consequently,
$$(1-q)\Theta_{p_0,q_0}(x_0)+q\Theta_{p_1,q_1}(x_1)
>(1-q)+qq_1=q.$$
Again equality forces $x_0=0$ and then $x_1=1$. Conversely, $x_0=0$ and $x_1=1$ give equality directly in every parameter range. The proof is complete.
\end{proof}

\subsection{\texorpdfstring{Proof of \cref{thm:shadow-main,cor:shadow-at-star}}
{Proof of the upper-shadow comparison}}

\begin{proof}[Proof of \cref{thm:shadow-main}]
Let $p=k/n$ and $q=\ell/n$. We argue by induction on $n$. The case
$k=\ell$ is immediate. By \cref{lem:balanced-restriction}, there is a
coordinate $i$ such that
$\mu_k(\F_0(i))\leq\mu_{k-1}(\F_1(i))$. Relabel it as $n$, and write
$$
\F_0=\{A\in\F:n\notin A\}\subseteq\binom{[n-1]}k,
\qquad
\F_1=\{A\setminus\{n\}:A\in\F,\ n\in A\}
\subseteq\binom{[n-1]}{k-1}.
$$
Then
$$
\mu_k(\F)=(1-p)\mu_k(\F_0)+p\mu_{k-1}(\F_1).
$$
Set $\F'=\partial_{k\to\ell}\F$, and define
$$
\F'_0=\{B\in\F':n\notin B\}\subseteq\binom{[n-1]}\ell,
\qquad
\F'_1=\{B\setminus\{n\}:B\in\F',\ n\in B\}
\subseteq\binom{[n-1]}{\ell-1}.
$$
Put
$$
p_0=\frac{k}{n-1},\quad q_0=\frac{\ell}{n-1},\qquad
p_1=\frac{k-1}{n-1},\quad q_1=\frac{\ell-1}{n-1}.
$$
Then
$$
\mu_\ell(\F')=(1-q)\mu_\ell(\F'_0)+q\mu_{\ell-1}(\F'_1),
$$
while
$\F'_0=\partial_{k\to\ell}\F_0$ and
$\F'_1\supseteq\partial_{k-1\to\ell-1}\F_1$ on $[n-1]$.
When $p_1>0$ and $q_0<1$, the induction hypothesis gives
$$
\mu_\ell(\F'_0)\geq
\Theta_{p_0,q_0}(\mu_k(\F_0)),
\qquad
\mu_{\ell-1}(\F'_1)\geq
\Theta_{p_1,q_1}(\mu_{k-1}(\F_1)).
$$
If $p_1=0$ or $q_0=1$, the corresponding inequality follows directly
from \eqref{eq:theta-profile-endpoints}.
Since $\mu_k(\F_0)\leq\mu_{k-1}(\F_1)$,
\cref{lem:slice-two-point} yields
$$
\begin{aligned}
\mu_\ell(\F')
&\geq(1-q)\Theta_{p_0,q_0}(\mu_k(\F_0))
+q\Theta_{p_1,q_1}(\mu_{k-1}(\F_1))\\
&\geq\Theta_{p,q}\bigl(
(1-p)\mu_k(\F_0)+p\mu_{k-1}(\F_1)
\bigr)=\Theta_{p,q}(\mu_k(\F)).
\end{aligned}
$$
This proves \cref{thm:shadow-main}.
\end{proof}

\begin{proof}[Proof of~\cref{cor:shadow-at-star}]
The inequality follows from~\cref{thm:shadow-main} and~\cref{lem:theta-comparison}(i).

Suppose that $k<\ell$, $\mu_k(\F)=p$, and equality holds in
\eqref{eq:uniform-power-comparison}. Set
$\F'=\partial_{k\to\ell}\F$. Since $\Theta_{p,q}(p)=q$, equality also
holds in \eqref{ineq:upper-shadow-theta}. In the inductive chain in the
proof of \cref{thm:shadow-main},
$$
q=\mu_\ell(\F')
\geq(1-q)\Theta_{p_0,q_0}(\mu_k(\F_0))
+q\Theta_{p_1,q_1}(\mu_{k-1}(\F_1))
\geq\Theta_{p,q}(p)=q.
$$
Hence equality holds in \eqref{ineq:slice-two-point}. Its equality
statement gives $\mu_k(\F_0)=0$ and
$\mu_{k-1}(\F_1)=1$. Therefore
$\F_0=\varnothing$ and
$\F_1=\binom{[n-1]}{k-1}$, so $\F$ is the full $1$-star centered at
the chosen coordinate. Conversely, a full $1$-star has density $p$
and its upper shadow is the corresponding full $1$-star of density
$q$.
\end{proof}

\subsection{\texorpdfstring{Proof of~\cref{thm:biased-shadow-comparison}}
{Proof of the biased-measure comparison}}
\label{subsec:product-limit}

\begin{proof}[Proof of~\cref{thm:biased-shadow-comparison}]
The case $p=q$ is immediate, since $\Theta_{p,p}(x)=x$ for every $x\in [0,1]$. The endpoint cases are also immediate. If $p=0<q$, then either $\varnothing\in\A$, in which case $\A=2^{[n]}$, or $\mu_0(\A)=0$; if $q=1$ and $p<1$, then an increasing family is nonempty if and only if it contains $[n]$. These alternatives agree exactly with \eqref{eq:theta-profile-endpoints}. We may therefore assume that $0<p<q<1$, and regard
$\A$ as an increasing family on a fixed ground set $[m]$. For every sufficiently large $N$, choose integers $m\leq K_N< L_N<N$ such that
$p_N:=\frac{K_N}{N}\to p$ and $q_N:=\frac{L_N}{N}\to q$ as $N\to \infty$. For $t\in \{K_N,L_N\}$, define
$$\widehat{\F}_N^{(t)}=\left\{A\in\binom{[N]}{t}:A\cap[m]\in\A\right\}.$$

We first claim that
\begin{equation}\label{eq:lifted-shadow-identity}
    \partial_{K_N\to L_N}\widehat{\F}_N^{(K_N)}=\widehat{\F}_N^{(L_N)}.
\end{equation}
If $B\in \partial_{K_N\to L_N}\widehat{\F}_N^{(K_N)}$, then $A\subseteq B$ for some $A\in \widehat{\F}_N^{(K_N)}$. Consequently, $A\cap [m]\subseteq B\cap [m]$ and the fact that $\A$ is increasing gives $B\cap [m]\in \A$. Thus $B\in\widehat{\F}_N^{(L_N)}$. Conversely, let $B\in \widehat{\F}_N^{(L_N)}$, and set $S=B\cap [m]\in \A$. Since $K_N\geq m\geq |S|$ and $K_N\leq L_N$, we may choose $T\in \binom{B\setminus [m]}{K_N-|S|}$, and set $A=S\cup T$. Then $A\subseteq B$, $|A|=K_N$, and $A\cap [m]=S\in \A$, proving the reverse inclusion.

We next compute the limiting densities of the families $\widehat{\F}_N^{(t)}$ for $t\in \{K_N,L_N\}$. Whenever $t/N\to \rho$,
$$\mu_t(\widehat{\F}_N^{(t)})=\sum_{S\in \A}\frac{\binom{N-m}{t-|S|}}{\binom{N}{t}}\to \sum_{S\in \A}\rho^{|S|}(1-\rho)^{m-|S|}=\mu_{\rho}(\A).$$
Hence,
$$\mu_{K_N}(\widehat{\F}_N^{(K_N)})\to \mu_p(\A), \quad \mu_{L_N}(\widehat{\F}_N^{(L_N)})\to \mu_q(\A).$$
Applying \cref{thm:shadow-main} to $\widehat{\F}_N^{(K_N)}$ and using \eqref{eq:lifted-shadow-identity}, we obtain
$$\mu_{L_N}\bigl(\widehat{\F}_N^{(L_N)}\bigr)
=\mu_{L_N}\left(\partial_{K_N\to L_N}\widehat{\F}_N^{(K_N)}\right)
\geq\Theta_{p_N,q_N}\left(\mu_{K_N}\bigl(\widehat{\F}_N^{(K_N)}\bigr)\right).$$
Letting $N\to\infty$ and using the joint continuity of $\Theta$ in the interior parameter range gives
$\mu_q(\A)\geq\Theta_{p,q}(\mu_p(\A))$.
The remaining inequality \eqref{eq:biased-power-comparison}, including the stated equality case, follows from \cite[Lemma 3.2]{CLL2026}.
\end{proof}

\subsection{Lower bounds for total influence}

For an increasing family $\A$, define its \emph{$p$-biased total influence} by
$$
{\mathrm I}^p(\A)=\sum_{i=1}^n\mathbb E\left[\mathbbm{1}_{\A}(A\cup\{i\})-\mathbbm{1}_{\A}(A)\right],
$$
where $A$ is a $p$-biased subset of $[n]\setminus\{i\}$. With this normalization, Margulis--Russo formula~\cite{Margulis1974,Russo1982} gives ${\mathrm I}^p(\A)=\frac{d}{dp}\mu_p(\A)$. Differentiating \cref{thm:biased-shadow-comparison} from the right at $q=p$ gives the following bound.

\begin{proposition}
\label{prop:infinitesimal-star-comparison}
Let $\A\subseteq2^{[n]}$ be increasing, let $0<p<1$, and set
$x=\mu_p(\A)$. Then
\begin{equation}\label{ineq:three-branch-influence}
{\mathrm I}^p(\A)\geq
\begin{cases}
\displaystyle
\frac{x\log x}{p\log p},
&0\leq x\leq p^2,\\[3mm]
\displaystyle
\frac{x}{p}+\frac{p-x}{1-p}
\frac{\log(1-x/p)}{\log(1-p)},
&p^2\leq x\leq p,\\[4mm]
\displaystyle
\frac{1-x}{1-p}\frac{\log(1-x)}{\log(1-p)},
&p\leq x\leq1.
\end{cases}
\end{equation}
Here $0\log0$ is interpreted as $0$, and the formulas at $p^2$ and
$p$ are interpreted by continuity. The adjacent expressions agree at
both breakpoints.
\end{proposition}

\begin{proof}
If $x=0$ or $x=1$, then $\A$ is $\varnothing$ or $2^{[n]}$ and the assertion is immediate. Assume that $0<x<1$. With $p$ fixed, \cref{thm:biased-shadow-comparison} and Margulis--Russo formula give
$$
{\mathrm I}^p(\A)=\lim_{q\downarrow p}\frac{\mu_q(\A)-\mu_p(\A)}{q-p}\geq
\left.\frac{\partial}{\partial q}\Theta_{p,q}(x)\right|_{q=p}.
$$
The three branch intervals in \eqref{eq:theta-profile} depend only on
$p$, so the derivative can be computed on each interval separately.
For $0<x\leq p^2$,
$$
\left.\frac{\partial}{\partial q}
x^{\frac{\log q}{\log p}}\right|_{q=p}
=
\left.x^{\frac{\log q}{\log p}}
\frac{\log x}{q\log p}\right|_{q=p}
=
\frac{x\log x}{p\log p}.
$$
For $p^2\leq x<p$,
$$
\begin{aligned}
\left.\frac{\partial}{\partial q}
q\left[1-\left(1-\frac{x}{p}\right)^{\frac{\log(1-q)}{\log(1-p)}}\right]\right|_{q=p}&=
\left.1-\left(1-\frac{x}{p}\right)^{\frac{\log(1-q)}{\log(1-p)}}+\frac{q\left(1-\frac{x}{p}\right)^{\frac{\log(1-q)}{\log(1-p)}}\log\left(1-\frac{x}{p}\right)}{(1-q)\log(1-p)}\right|_{q=p}\\
&=\frac{x}{p}+\frac{p-x}{1-p}
\frac{\log\left(1-\frac{x}{p}\right)}{\log(1-p)}.
\end{aligned}
$$
Finally, for $p\leq x<1$,
$$
\left.\frac{\partial}{\partial q}
\left[1-(1-x)^{\frac{\log(1-q)}{\log(1-p)}}\right]\right|_{q=p}=\left.
\frac{(1-x)^{\frac{\log(1-q)}{\log(1-p)}}\log(1-x)}{(1-q)\log(1-p)}\right|_{q=p}=\frac{1-x}{1-p}\frac{\log(1-x)}{\log(1-p)}.
$$
Substitution into the three derivatives gives~\eqref{ineq:three-branch-influence}. At $x=p^2$, the first two expressions equal $2p$, and at $x=p$, the last two equal $1$.
\end{proof}

The same families that attain the three branches of~\eqref{ineq:biased-theta} attain the corresponding bounds in~\eqref{ineq:three-branch-influence}. The families $\{A\subseteq[n]:S\subseteq A\}$ with $|S|\geq2$ attain the first
line. The families $\{A\subseteq[n]:j\in A,\ A\cap S\neq\varnothing\}$ with $j\notin S$ and $S\neq\varnothing$ attain the second line, and the families $\{A\subseteq[n]:A\cap S\neq\varnothing\}$ with $S\neq\varnothing$ attain the third line.

\begin{remark}
For comparison, we recall the exact biased edge-isoperimetric theorem. Let $\mu_p^{(\mathbb N)}$ be Bernoulli product measure on $2^{\mathbb N}$, and let ${\mathrm I}_p^{(\mathbb N)}$ denote the sum of the coordinate influences on this infinite product cube. We use on $2^{\mathbb N}$ the \emph{lexicographic order} already defined on the finite slices: for distinct $A,B\subseteq\mathbb N$, we put $A<_{\rm lex}B$ when the least element of $A\mathbin{\triangle}B$ belongs to $A$. For $\lambda\in[0,1]$, let $\mathcal L_\lambda$ be the measurable initial segment of this order with $\mu_{1/2}^{(\mathbb N)}(\mathcal L_\lambda)=\lambda$.  Given $0<p<1$ and $x\in[0,1]$, choose $\lambda_p(x)$ so that
$\mu_p^{(\mathbb N)}(\mathcal L_{\lambda_p(x)})=x$. The exact biased edge-isoperimetric profile is
\begin{equation}\label{eq:exact-biased-profile}
\mathcal I_p(x):=\inf_{\substack{\A\subseteq2^{\mathbb N}\ \text{measurable and increasing}\\
\mu_p^{(\mathbb N)}(\A)=x}}
{\mathrm I}_p^{(\mathbb N)}(\A)={\mathrm I}_p^{(\mathbb N)}(\mathcal L_{\lambda_p(x)}).
\end{equation}
The equality in~\eqref{eq:exact-biased-profile} is the biased
edge-isoperimetric theorem of Ellis, Keller and
Lifshitz~\cite[Theorem~1.9]{EKL2019JCTA}. In particular, every
increasing $\A\subseteq2^{[n]}$ with $\mu_p(\A)=x$ satisfies
${\mathrm I}^p(\A)\geq\mathcal I_p(x)$. \cref{prop:infinitesimal-star-comparison} gives instead the elementary
three-branch lower bound in \eqref{ineq:three-branch-influence}. It is
sharp for the three classes of families above, but for a general
value of $x$ it need not equal the lexicographic quantity in
\eqref{eq:exact-biased-profile}.
\end{remark}

\subsection{Composition and families attaining the bounds}
\label{sec:profile-geometry}

Upper-shadow operators compose exactly. Indeed, let $1\leq k\leq\ell\leq h\leq n$ and $\F\subseteq\binom{[n]}k$. For every $B\in\binom{[n]}h$,
$$
\begin{aligned}
B\in\partial_{\ell\to h}\bigl(\partial_{k\to\ell}\F\bigr)
&\Longleftrightarrow
\text{there exist }A\in\F\text{ and }C\in\binom{[n]}\ell\text{ such that }A\subseteq C\subseteq B\\
&\Longleftrightarrow
\text{there exists }A\in\F\text{ such that }A\subseteq B\\
&\Longleftrightarrow
B\in\partial_{k\to h}\F.
\end{aligned}
$$
For the reverse implication in the second equivalence, if $A\subseteq B$, choose $D\in\binom{B\setminus A}{\ell-k}$ and set $C=A\cup D$. Consequently,
$$
\partial_{\ell\to h}\bigl(\partial_{k\to\ell}\F\bigr)
=\partial_{k\to h}\F.
$$

The next lemma is the corresponding composition identity for $\Theta$.

\begin{lemma}[Composition identity]\label{lem:profile-semigroup}
Let $0<p\leq q\leq r<1$. For every $x\in[0,1]$,
$$
\Theta_{q,r}\bigl(\Theta_{p,q}(x)\bigr)
=\Theta_{p,r}(x).
$$
\end{lemma}

\begin{proof}
Direct substitution in the three formulas for $\Theta_{p,q}$ gives
$$\Theta_{p,q}([0,p^2])=[0,q^2],\quad
\Theta_{p,q}([p^2,p])=[q^2,q],\quad
\Theta_{p,q}([p,1])=[q,1].$$
Hence, under the composition $\Theta_{q,r}\circ\Theta_{p,q}$, each of the three branch intervals for $\Theta_{p,q}$ is mapped into the corresponding branch interval for $\Theta_{q,r}$. We therefore consider the three cases separately.

If $0\leq x\leq p^2$, then $\Theta_{p,q}(x)\in[0,q^2]$, and hence
$$
\Theta_{q,r}(\Theta_{p,q}(x))
=\left(x^{\frac{\log q}{\log p}}\right)^{
\frac{\log r}{\log q}}
=x^{\frac{\log r}{\log p}}.
$$
If $p^2\leq x\leq p$, then $\Theta_{p,q}(x)\in[q^2,q]$ and
$$
\begin{aligned}
\Theta_{q,r}(\Theta_{p,q}(x))
=r\left[1-\left(1-\frac{\Theta_{p,q}(x)}q\right)^{
\frac{\log(1-r)}{\log(1-q)}}\right]=r\left[1-\left(1-\frac{x}{p}\right)^{
\frac{\log(1-r)}{\log(1-p)}}\right].
\end{aligned}
$$
Finally, if $p\leq x\leq1$, then $\Theta_{p,q}(x)\in[q,1]$ and
$$
\Theta_{q,r}(\Theta_{p,q}(x))
=1-\left((1-x)^{\frac{\log(1-q)}{\log(1-p)}}\right)^{
\frac{\log(1-r)}{\log(1-q)}}
=1-(1-x)^{\frac{\log(1-r)}{\log(1-p)}}.
$$
These are precisely the three formulas defining $\Theta_{p,r}(x)$.
\end{proof}

Thus successive applications of \eqref{ineq:upper-shadow-theta} for
$(k,\ell)$ and $(\ell,h)$ recover the inequality for $(k,h)$.

We now give three families for which \cref{thm:shadow-main} is
asymptotically sharp. Fix an integer $t\geq1$.  For $n\geq t+1$, define
$$
\begin{aligned}
&\mathcal S_{n,k}^{(t)}=\left\{A\in\binom{[n]}k:[t]\subseteq A\right\}, \quad  \mathcal V_{n,k}^{(t)}=\left\{A\in\binom{[n]}k:1\in A,\ A\cap\{2,\dots,t+1\}\neq\varnothing\right\}, \\
&\mathcal U_{n,k}^{(t)}=\left\{A\in\binom{[n]}k:A\cap[t]\neq\varnothing\right\}.
\end{aligned}
$$
These are initial segments of the lexicographic order on $\binom{[n]}{k}$. Moreover, $\partial_{k\to\ell}\mathcal{S}_{n,k}^{(t)}=\mathcal{S}_{n,\ell}^{(t)}$ for $t\leq k\leq\ell$, $\partial_{k\to\ell}\mathcal{V}_{n,k}^{(t)}=\mathcal{V}_{n,\ell}^{(t)}$ for $2\leq k\leq\ell$, and $\partial_{k\to\ell}\mathcal{U}_{n,k}^{(t)}=\mathcal{U}_{n,\ell}^{(t)}$ for $1\leq k\leq\ell$.

\begin{proposition}[Asymptotic sharpness]\label{prop:sharpness-families}
Fix an integer $t\geq1$ and real numbers $0<p\leq q<1$. Let $k=k(n)\leq\ell=\ell(n)$ satisfy
$p_n=\frac{k}{n}\to p$ and $q_n=\frac{\ell}{n}\to q$. Then
$$\mu_\ell(\partial_{k\to\ell}\mathcal{S}_{n,k}^{(t)})-\Theta_{p_n,q_n}(\mu_k(\mathcal{S}_{n,k}^{(t)}))\to0,\quad\mu_\ell(\partial_{k\to\ell}\mathcal{V}_{n,k}^{(t)})-\Theta_{p_n,q_n}(\mu_k(\mathcal{V}_{n,k}^{(t)}))\to0,$$
and
$$
\mu_\ell(\partial_{k\to\ell}\mathcal{U}_{n,k}^{(t)})-\Theta_{p_n,q_n}(\mu_k(\mathcal{U}_{n,k}^{(t)}))\to0.
$$
Thus $\mathcal{S}_{n,k}^{(t)}$ for $t\geq2$, $\mathcal{V}_{n,k}^{(t)}$ for $t\geq1$, and $\mathcal{U}_{n,k}^{(t)}$ for $t\geq1$ asymptotically attain the first,
middle, and third branches, respectively.
\end{proposition}

\begin{proof}
The three $k$-slice densities are $\frac{\binom{n-t}{k-t}}{\binom nk}$, $\frac{k}{n}\left(1-\frac{\binom{n-t-1}{k-1}}{\binom{n-1}{k-1}}\right)$, and $1-\frac{\binom{n-t}{k}}{\binom nk}$, respectively. Their limits are $p^t$, $p\bigl(1-(1-p)^t\bigr)$, and $1-(1-p)^t$, respectively.
The corresponding limits on the $\ell$-th level are $q^t$, $q(1-(1-q)^t)$, and $1-(1-q)^t$.  By the three branches of \eqref{eq:theta-profile},
$$\Theta_{p,q}(p^t)=q^t, \quad \Theta_{p,q}\bigl(p(1-(1-p)^t)\bigr)=q(1-(1-q)^t), \quad \Theta_{p,q}\bigl(1-(1-p)^t\bigr)=1-(1-q)^t.$$
The exact shadow identities above and continuity of $\Theta_{p,q}$ in its parameters and argument prove the three limits.
\end{proof}

\subsection{The Erd\H{o}s--Ko--Rado theorem}

As a direct application of \cref{cor:shadow-at-star}, we recover the classical Erd\H{o}s--Ko--Rado theorem.

\begin{corollary}[Erd\H{o}s--Ko--Rado \cite{EKR1961}]\label{cor:ekr-from-shadow}
Let $1\leq k\leq \frac{n}{2}$, and let
$\F\subseteq\binom{[n]}k$ be intersecting. Then
$$\mu_k(\F)\leq\frac{k}{n}.$$
If $k<n/2$, equality holds if and only if $\F$ is a $1$-star.
\end{corollary}

\begin{proof}
Let $p=\frac{k}{n}$ and
$$\overline{\F}=\left\{[n]\setminus A:A\in\F\right\}
\subseteq\binom{[n]}{n-k}.$$
Since $\F$ is intersecting, we have
$$\partial_{k\to n-k}\F\cap\overline{\F}=\varnothing.$$
Indeed, otherwise there would be $A,A'\in\F$ such that
$A\subseteq[n]\setminus A'$, contrary to $A\cap A'\neq\varnothing$.
Moreover,
$\mu_{n-k}(\overline{\F})=\mu_k(\F)$,
and hence
$$\mu_{n-k}(\partial_{k\to n-k}\F)\leq 1-\mu_k(\F).$$
On the other hand, applying \cref{cor:shadow-at-star} with
$\ell=n-k$, gives
$$\mu_{n-k}(\partial_{k\to n-k}\F)\geq \mu_k(\F)^{\frac{\log(1-p)}{\log p}}.$$
Consequently,
$$\mu_k(\F)+\mu_k(\F)^{\frac{\log(1-p)}{\log p}}\leq1.$$
The left-hand side is strictly increasing in $\mu_k(\F)$, and its value at $\mu_k(\F)=p$ is $p+p^{\frac{\log(1-p)}{\log p}}=p+(1-p)=1$. Therefore $\mu_k(\F)\leq p$.

Suppose now that $k<\frac{n}{2}$ and $\mu_k(\F)=p$. Then the upper and lower estimates for the shadow are both equalities, so
$$
\mu_{n-k}(\partial_{k\to n-k}\F)=1-p.
$$
Since $k<n-k<n$ and $\mu_k(\F)=p$, applying \cref{cor:shadow-at-star} with $\ell=n-k$ shows that $\F$ is a $1$-star. This also proves the converse,
because a $1$-star is intersecting and has density $p$.
\end{proof}

\begin{remark}[Relation to earlier shadow proofs]\label{rem:ekr-history}
The use of complements and shadows in the Erd\H{o}s--Ko--Rado theorem
is classical. Daykin~\cite{Daykin1974} deduced the theorem directly
from Kruskal--Katona, and Frankl and F\"uredi~\cite{FranklFuredi2012}
observe that shadows of complements already appear in Katona's
work~\cite[p.~334]{Katona1964}. In \cref{cor:ekr-from-shadow}, the
exact Kruskal--Katona expansion is replaced by the closed-form estimate
\eqref{eq:uniform-power-comparison}, reducing the proof to
$$x+x^{\frac{\log(1-p)}{\log p}}\leq1.$$
This is the simplest instance of the same upper-shadow comparison
used for the Frankl--Tokushige product bounds.
\end{remark}

\section{Proofs of the Frankl--Tokushige conjectures}\label{sec:uniform-proof}

\subsection{Critical-sum inequalities}

We begin with the coupling inequalities used in both the uniform and
biased proofs.

\begin{lemma}\label{lem:uniform-critical-sum}
Let $0\leq\ell_i\leq n$ be integers satisfying $\sum_{i=1}^r\ell_i=(r-1)n$. If $\F_i\subseteq\binom{[n]}{\ell_i}$ are $r$-cross-intersecting, then
$$
\sum_{i=1}^r\mu_{\ell_i}(\F_i)\leq r-1.
$$
\end{lemma}

\begin{proof}
Choose a uniformly random ordered partition $[n]=P_1\mathbin{\dot\cup}\cdots\mathbin{\dot\cup}P_r$ with $|P_i|=n-\ell_i$, and set $X_i=[n]\setminus P_i$. Then $X_i$ is uniformly distributed on $\binom{[n]}{\ell_i}$ and
$$
X_1\cap\cdots\cap X_r=\varnothing.
$$
Thus the events $X_i\in\F_i$ cannot all occur. Hence
$$\sum_{i=1}^r\mathbbm{1}_{\{X_i\in\F_i\}}\leq r-1,$$
and taking expectations proves the lemma.
\end{proof}

The same construction has the following product-measure form. When
$q_1=\cdots=q_r=(r-1)/r$, it is the equal-bias coupling from
\cite[Lemma~3.1]{CLL2026}.
\begin{lemma}\label{lem:biased-critical-sum}
Let $q_1,\ldots,q_r\in[0,1]$ satisfy $\sum_iq_i=r-1$. If
$\A_1,\ldots,\A_r\subseteq2^{[n]}$ are $r$-cross-intersecting, then
$$
\sum_{i=1}^r\mu_{q_i}(\A_i)\leq r-1.
$$
\end{lemma}

\begin{proof}
Independently for each coordinate $x\in[n]$, choose an index
$I_x\in[r]$ with $\Pr[I_x=i]=1-q_i$. Define $X_i\subseteq[n]$ by
$x\in X_i$ if and only if $I_x\neq i$. For each $x\in[n]$,
$$\Pr[x\in X_i]=1-\Pr[I_x=i]=q_i.$$
Therefore $X_i\sim \mu_{q_i}$ for every $i\in[r]$. On the other hand, every coordinate $x\in[n]$ is omitted from exactly one of the sets $X_1,\dots,X_r$, hence
$$X_1\cap\cdots\cap X_r=\varnothing.$$
Since $\A_1,\dots,\A_r$ are $r$-cross-intersecting, the events $X_i\in\A_i$ cannot all occur. Hence $$\sum_{i=1}^r\mathbbm{1}_{\{X_i\in\A_i\}}\leq r-1.$$
Taking expectations and using $X_i\sim\mu_{q_i}$ gives
$\sum_{i=1}^r\mu_{q_i}(\A_i)\leq r-1$.
\end{proof}

The equality classifications in
\cref{thm:uniform-main,thm:biased-main} require the corresponding
equality cases of the two coupling inequalities.

\begin{lemma}\label{lem:uniform-partition-equality}
Let $r\geq3$ and $k=\frac{n(r-1)}{r}$. If $\F_1,\dots,\F_r\subseteq \binom{[n]}{k}$ are $r$-cross-intersecting and $\sum_{i=1}^r\mu_k(\F_i)=r-1$, then either there is $x\in[n]$ such that
$$\F_i=\left\{A\in\binom{[n]}k:x\in A\right\},$$
for every $i\in[r]$, or one of the $\F_i$ is empty and all the others are $\binom{[n]}{k}$.
\end{lemma}

\begin{proof}
For every $i\in[r]$, let $\F_i^*$ be the uniform dual family of $\F_i$. Then $\mu_{n-k}(\F_i^*)=1-\mu_{k}(\F_i)$. By \cref{lem:uniform-dual-family}, every ordered partition $[n]=P_1\mathbin{\dot\cup}\cdots\mathbin{\dot\cup}P_r$ into
$(n-k)$-sets has at least one index $i$ with $P_i\in \F_i^*$. Choosing such a partition uniformly at random, we have,
$$\mathbb{E}\left[\sum_{i=1}^r\mathbbm{1}_{\{P_i\in\F_i^*\}}\right]=\sum_{i=1}^r\mu_{n-k}(\F_i^*)=1. $$
Hence \begin{equation}\label{eq:uniform-sum-one}
    \sum_{i=1}^r\mathbbm{1}_{\{P_i\in\F_i^*\}}=1.
\end{equation}

Fix distinct $i,j\in[r]$, and let $A,B\in\binom{[n]}{n-k}$ be disjoint.
Partition $[n]\setminus(A\cup B)$ into $(n-k)$-sets indexed by
$[r]\setminus\{i,j\}$, and set $P_i=A$ and $P_j=B$. Applying \eqref{eq:uniform-sum-one} to this ordered partition and to the partition obtained by interchanging $A$ and $B$, and subtracting the two identities
gives
$$
\mathbbm{1}_{\{A\in\F_i^*\}}-
\mathbbm{1}_{\{A\in\F_j^*\}}
=
\mathbbm{1}_{\{B\in\F_i^*\}}-
\mathbbm{1}_{\{B\in\F_j^*\}}.
$$
For arbitrary $A,B\in\binom{[n]}{n-k}$, there is an $(n-k)$-set $C$ disjoint from $A\cup B$, since $n-|A\cup B|\geq n-2(n-k)\geq n-k$. Applying the
preceding identity to $(A,C)$ and $(B,C)$ shows that
$\mathbbm{1}_{\F_i^*}-\mathbbm{1}_{\F_j^*}=c_{ij}$ on $\binom{[n]}{n-k}$, for some $c_{ij}\in \{-1,0,1\}$.

Suppose first that some $c_{ij}$ is nonzero. After interchanging $i$ and $j$, if necessary, assume that $c_{ij}=1$. Then $\F_i^*=\binom{[n]}{n-k}$ and $\F_j^*=\varnothing$. Since $\F_i^*$ contributes a selected block to every ordered partition, \eqref{eq:uniform-sum-one} forces $\F_h^*=\varnothing$ for every $h\neq i$. Hence  $\F_i=\varnothing$ and $\F_h=\binom{[n]}{k}$ for all $h\neq i$.

It remains to consider the case in which every $c_{ij}=0$. Then $\F_1^*=\cdots=\F_r^*=:\F^*$ and $\F_1=\cdots=\F_r=:\F$. Since $\sum_{i=1}^r\mu_k(\F_i)=r-1$, we have $\sum_{i=1}^r\mu_{n-k}(\F_i^*)=1$ and $\mu_{n-k}(\F^*)=\frac{1}{r}$. We claim that the family $\F^*$ is intersecting. Indeed, if
$A,B\in\F^*$ with $A\cap B=\varnothing$, an ordered partition with $P_1=A$ and $P_2=B$ would have at least two selected blocks, contradicting \eqref{eq:uniform-sum-one}. Moreover, $|\F^*|=\frac{1}{r}\binom{n}{n-k}=\binom{n-1}{n-k-1}$.
Since $n=r(n-k)>2(n-k)$, the equality statement in \cref{cor:ekr-from-shadow} shows that $\F^*$ is a full $1$-star. Hence $\F$ is the same full $1$-star, as needed.
\end{proof}

\begin{lemma}\label{lem:biased-equality}
Let $r\geq3$ and $p=\frac{r-1}{r}$. If $\A_1,\dots,\A_r\subseteq 2^{[n]}$ are increasing and $r$-cross-intersecting and $\sum_{i=1}^r\mu_p(\A_i)=r-1$, then either there is $x\in[n]$ such that
$$\A_i=\left\{A\subseteq [n]:x\in A\right\},$$
for every $i\in[r]$, or one of the $\A_i$ is empty and all the others are $2^{[n]}$.
\end{lemma}

\begin{proof}
For each $i\in[r]$, let $\A_i^*$ be the dual family of $\A_i$, that is,
  $$\A_i^*:=\{B\subseteq[n]:[n]\setminus B\notin\A_i\}.$$
Then $\mu_{1/r}(\A_i^*)=1-\mu_p(\A_i)$ for all $i\in[r]$, and therefore $\sum_{i=1}^r\mu_{1/r}(\A_i^*)=1$. By \cref{lem:biased-dual-family}, every ordered partition $[n]=P_1\mathbin{\dot\cup}\cdots\mathbin{\dot\cup}P_r$ has at least one index $i$ with $P_i\in \A_i^*$. Choose a random ordered partition by assigning each element of $[n]$ independently and uniformly to one of the $r$ parts. Then each $P_i$ has distribution $\mu_{1/r}$, and hence
$$\mathbb{E}\left[\sum_{i=1}^r\mathbbm{1}_{\{P_i\in\A_i^*\}}\right]=\sum_{i=1}^r\mu_{1/r}(\A_i^*)=1. $$
Hence \begin{equation}\label{eq:biased-sum-equality}
    \sum_{i=1}^r\mathbbm{1}_{\{P_i\in\A_i^*\}}=1.
\end{equation}

Fix distinct $i,j\in[r]$ and disjoint sets $A,B\subseteq[n]$. Since $r\geq3$, complete $A$ and $B$ to an ordered partition with $P_i=A$ and $P_j=B$. Comparing \eqref{eq:biased-sum-equality} for this partition and for the partition obtained by interchanging $A$ and $B$ gives
$$
\mathbbm{1}_{\{A\in\A_i^*\}}-
\mathbbm{1}_{\{A\in\A_j^*\}}
=
\mathbbm{1}_{\{B\in\A_i^*\}}-
\mathbbm{1}_{\{B\in\A_j^*\}}.
$$
Taking $B=\varnothing$, we obtain
$\mathbbm{1}_{\A_i^*}-\mathbbm{1}_{\A_j^*}=c_{ij}$ on $2^{[n]}$, for some $c_{ij}\in \{-1,0,1\}$.

Suppose first that some $c_{ij}$ is nonzero. After interchanging $i$ and $j$, assume $c_{ij}=1$. Then $\A_i^*=2^{[n]}$ and $\A_j^*=\varnothing$. Since $\A_i^*$ contributes a selected block to every ordered partition, \eqref{eq:biased-sum-equality} forces $\A_h^*=\varnothing$ for every $h\neq i$. Then $\A_i=\varnothing$ and $\A_h=2^{[n]}$ for all $h\neq i$.

It remains to consider the case $c_{ij}=0$ for every pair $i,j$. Then $\A_1^*=\cdots=\A_r^*=:\A^*$ and $\A_1=\cdots=\A_r=:\A$. We claim that $\A^*$ is a full $1$-star. Fix $x\in[n]$ and $A\subseteq[n]\setminus\{x\}$. Apply \eqref{eq:biased-sum-equality}  to the two ordered partitions
\[
  \bigl(A\cup\{x\},\,\varnothing,\,[n]\setminus(A\cup\{x\}),\,
  \varnothing,\dots,\varnothing\bigr)
\]
and
\[
  \bigl(A,\,\{x\},\,[n]\setminus(A\cup\{x\}),\,
  \varnothing,\dots,\varnothing\bigr).
\]
Subtracting the resulting identities yields
\[
  \mathbbm{1}_{\A^*}(A\cup\{x\})-\mathbbm{1}_{\A^*}(A)
  =\mathbbm{1}_{\A^*}(\{x\})-\mathbbm{1}_{\A^*}(\varnothing)
  =:d_x.
\]
Since $\mathbbm{1}_{\A^*}$ is $\{0,1\}$-valued, $d_x\in\{-1,0,1\}$. If $d_x=1$, then
$\mathbbm{1}_{\A^*}(A)=0$ and $\mathbbm{1}_{\A^*}(A\cup\{x\})=1$ for every $A\subseteq[n]\setminus\{x\}$, and hence
\[
  \A^*=\{A\subseteq[n]:x\in A\}.
\]
If $d_x=-1$, then $\A^*=\{A\subseteq[n]:x\notin A\}$. In that case every ordered partition has exactly $r-1$ parts belonging to $\A^*$, contradicting \eqref{eq:biased-sum-equality} because $r\geq3$. Consequently, no $d_x$ equals $-1$. If no $d_x$ equals $1$, then $d_x=0$ for every $x$, so $\mathbbm{1}_{\A^*}$ is constant on $2^{[n]}$. Every ordered partition would then have either no selected part or all $r$ parts selected, again contradicting \eqref{eq:biased-sum-equality}. Therefore $d_x=1$ for some $x\in[n]$, and $\A^*$ is the full
$1$-star centered at $x$. Since a full $1$-star is self-dual,
$\A=\A^*$ is the same full $1$-star.
\end{proof}

As a first application, we recover the equal-level theorem of Frankl
and Tokushige~\cite{FT2011} by combining
\cref{lem:uniform-critical-sum} with
\eqref{eq:uniform-power-comparison}.
\begin{theorem}[Frankl--Tokushige~\cite{FT2011}]\label{thm:equal-level-theorem}
    Suppose that $\F_1,\ldots,\F_r\subseteq\binom{[n]}k$ are $r$-cross-intersecting and $rk\leq(r-1)n$. Then
    $$\prod_{i=1}^r|\F_i|\leq \binom{n-1}{k-1}^{r}.$$
\end{theorem}

\begin{proof}
We may assume that all the families are nonempty and $1\leq k\leq \frac{r-1}{r}n$. Write $n=ar+b$ for $0\leq b<r$. The assumption $rk\leq(r-1)n$ implies $k\leq n-a$ when $b=0$, and $k\leq n-a-1$ when $b>0$.

For each choice of $b$ indices from $[r]$, replace $\F_i$ by its upper shadow on level
$n-a-1$ for the chosen indices, and by its upper shadow on level $n-a$ for the remaining
indices. These upper shadows remain $r$-cross-intersecting. Moreover, the sum
of their levels is
$$
b(n-a-1)+(r-b)(n-a)=(r-1)n.
$$
\cref{lem:uniform-critical-sum} therefore gives an additive bound for every such
choice.

Averaging this bound over all choices of the $b$ indices gives
$$
\frac{b}{r}\sum_{i=1}^r\mu_{n-a-1}\left(\partial_{k\to n-a-1}\F_i
\right)+\frac{r-b}{r}\sum_{i=1}^r
\mu_{n-a}\left(\partial_{k\to n-a}\F_i\right)\leq r-1.
$$
When $b=0$, the first sum is omitted. If $a=0$, then the second upper level is $n$;
since every $\F_i$ is nonempty, $\partial_{k\to n}\F_i=\{[n]\}$,
so the corresponding shadow density is $1$.

Applying~\eqref{eq:uniform-power-comparison}, followed by AM--GM, yields
$$
\begin{aligned}
r-1&\geq\frac{b}{r}\sum_{i=1}^r
\mu_k(\F_i)^{\frac{\log((n-a-1)/n)}{\log(k/n)}}+\frac{r-b}{r}
\sum_{i=1}^r\mu_k(\F_i)^{
\frac{\log((n-a)/n)}{\log(k/n)}}\\
&\geq b\left(\prod_{i=1}^r\mu_k(\F_i)\right)^{\frac{\log((n-a-1)/n)}{r\log(k/n)}}+(r-b)\left(\prod_{i=1}^r\mu_k(\F_i)\right)^{\frac{\log((n-a)/n)}{r\log(k/n)}}.
\end{aligned}
$$
As above, the first term is omitted when $b=0$, and an exponent with numerator
$\log 1$ is interpreted as $0$.

The expression in the last line is an increasing function of $\prod_{i=1}^r\mu_k(\F_i)$. In fact, it is strictly increasing, since at least one of the upper levels occurring with
positive multiplicity is smaller than $n$. When $\F_1,\dots,\F_r$ are the corresponding levels of a common $1$-star, i.e., $\prod_{i=1}^r\mu_k(\F_i)=
\left(\frac{k}{n}\right)^r$, this expression is exactly
$$
b\left(\left(\frac{k}{n}\right)^r\right)^{\frac{\log((n-a-1)/n)}{r\log(k/n)}}+(r-b)\left(\left(\frac{k}{n}\right)^r\right)^{\frac{\log((n-a)/n)}{r\log(k/n)}}=\frac{b(n-a-1)+(r-b)(n-a)}{n}=r-1.
$$
It follows directly from the preceding inequality and monotonicity that
$$
\prod_{i=1}^r\mu_k(\F_i)\leq\left(\frac{k}{n}\right)^r.
$$
Consequently,
$$
\prod_{i=1}^r|\F_i|=\binom{n}{k}^r\prod_{i=1}^r\mu_k(\F_i)\leq
\left(\frac{k}{n}\binom{n}{k}
\right)^r=\binom{n-1}{k-1}^r.
$$
This recovers the Frankl--Tokushige equal-level theorem without passing to $r$-cross-union families. The $r$-cross-intersecting condition is used only through \cref{lem:uniform-critical-sum}.
\end{proof}

\begin{remark}[Use of \eqref{eq:theta-profile} in the intersection bounds]
\label{rem:branch-roles}
The proof of~\cref{cor:ekr-from-shadow} and the preceding equal-level
argument use only~\eqref{eq:uniform-power-comparison}, together with~\cref{lem:uniform-critical-sum} and AM--GM in the latter case.

For unequal levels, let $j$ and $q$ be as in the proof of
\cref{thm:uniform-main}, write $\mu_{k_j}(\F_j)=p_ja_j$, and put
$$
\alpha=\frac{\log q}{\log p_j},
\qquad
\beta=\frac{\log(1-q)}{\log(1-p_j)}.
$$
Then \cref{thm:shadow-main} gives
$$
\mu_{qn}(\partial_{k_j\to qn}\F_j)-q
\geq
\begin{cases}
q(a_j^\alpha-1),&0<a_j\leq p_j,\\[1mm]
-q(1-a_j)^\beta,&p_j\leq a_j\leq1,\\[1mm]
(1-p_j)^\beta-(1-p_ja_j)^\beta,
&1\leq a_j\leq \frac{1}{p_j}.
\end{cases}
$$
On $a_j\in(0,p_j]$ and $a_j\in[p_j,1)$, the first two lines combine with~\eqref{ineq:uniform-critical-sum-constraint} to give
hypotheses~(i) and~(ii) of \cref{thm:product-inequality}. On
$a_j\in[1,\frac{1}{p_j}]$, the third line and the same constraint give
$a_i=1$ for $i\leq j$, and the ordering of the $a_i$ completes the
product bound. Thus each of the three formulas in~\eqref{eq:theta-profile} is used in the unequal-level proof.

The middle formula cannot in general be replaced by~\eqref{eq:uniform-power-comparison}. For $r=2$, the latter and the critical-sum inequality give only
$$
p_1(a_1-1)\leq q(1-a_2^\alpha),
\qquad
\alpha=\frac{\log q}{\log p_2}.
$$
For example, $p_1=\frac{1}{4}$, $p_2=\frac{1}{2}$, $a_1=\frac{59}{50}$, and $a_2=\frac{43}{50}$ satisfy this scalar inequality, whereas $a_1a_2=1.0148>1$. This example is not asserted to arise from cross-intersecting families; it shows that~\eqref{eq:uniform-power-comparison} alone does not imply the required product bound.
\end{remark}

\subsection{\texorpdfstring{Proof of~\cref{thm:uniform-main} and~\cref{thm:biased-main}}
{Proof of the uniform and biased theorems}}

The remaining step is the following analytic inequality.
\begin{theorem}\label{thm:product-inequality}
Let $r\geq2$, and let $0<p_i\leq \frac{r-1}{r}$ for every $i\in[r]$.
Suppose that $a_1\geq\cdots\geq a_r>0$. Let $j$ be the first index
such that $\sum_{i=1}^{j}(1-p_i)\geq1$, put $q=\sum_{i<j}(1-p_i)$, and define $\alpha=\frac{\log q}{\log p_j}$ and $\beta=\frac{\log(1-q)}{\log(1-p_j)}$. Suppose that one of the following holds:
\begin{itemize}
    \item[(i)] $0<a_j\leq p_j$ and $\sum_{i<j}p_i(a_i-1)\leq q(1-a_j^\alpha)$;
    \item[(ii)] $p_j\leq a_j<1$ and $\sum_{i<j}p_i(a_i-1)\leq q(1-a_j)^\beta$.
\end{itemize}
Then
$$
\prod_{i=1}^{r}a_i< 1.
$$
\end{theorem}

The proof of~\cref{thm:product-inequality} is purely analytic and is deferred to~\cref{sec:product-inequality}. We first use it to complete the proof of~\cref{thm:uniform-main}.

\begin{proof}[Proof of inequality in~\cref{thm:uniform-main}]
If some $\F_i$ is empty, the conclusion is immediate. If some $k_i=0$, then either $\F_i$ is empty or $\F_i=\{\varnothing\}$. In the latter case, cross-intersection forces another family to be empty. We may therefore assume that every $k_i$ and every $\mu_{k_i}(\F_i)$ is positive.

Set $p_i=\frac{k_i}{n}$ and $a_i=\frac{\mu_{k_i}(\F_i)}{p_i}$. Relabel the families so that
$$a_1\geq\cdots\geq a_r.$$
We prove that $\prod_i a_i\leq1$.

Let $j$ be the first index such that $\sum_{i=1}^j(1-p_i)\geq1$, and put $q=\sum_{i<j}(1-p_i)$.
Such an index exists because $1-p_i\geq1/r$ for every $i$. Moreover, $q<1$, $q\geq p_j$, and $qn$ is an integer. Since $\F_1,\dots,\F_r$ are $r$-cross-intersecting, the families
$$\F_1,\ldots,\F_{j-1},\partial_{k_j\to qn}\F_j,\binom{[n]}{n},\ldots,\binom{[n]}{n}$$
are also $r$-cross-intersecting. Moreover, $$\sum_{i<j}k_i+qn+(r-j)n=(r-1)n.$$
Applying \cref{lem:uniform-critical-sum},
$$\sum_{i<j}\mu_{k_i}(\F_i)+\mu_{qn}(\partial_{k_j\to qn}\F_j)+(r-j)\leq r-1,$$
which implies
\begin{equation}\label{ineq:uniform-critical-sum-constraint}
\sum_{i<j}\left(\mu_{k_i}(\F_i)-p_i\right)+\mu_{qn}(\partial_{k_j\to qn}\F_j)-q\leq 0.
\end{equation}
Set $\alpha=\frac{\log q}{\log p_j}$ and $\beta=\frac{\log(1-q)}{\log(1-p_j)}.$

\medskip
\noindent\textbf{Case 1: $a_j\geq 1$.} Then $a_i\geq1$ for every $i<j$, and the third branch of \cref{thm:shadow-main} gives
$$\mu_{qn}(\partial_{k_j\to qn}\F_j)-q\geq 1-q-(1-p_ja_j)^{\beta}=(1-p_j)^{\beta}-(1-p_ja_j)^{\beta}\geq 0.
$$
Every term on the left-hand side of~\eqref{ineq:uniform-critical-sum-constraint} is nonnegative. Hence $a_i=1$ for every $i\leq j$. The ordering gives $a_i\leq1$ for every $i>j$, and therefore $\prod_i a_i\leq1$.

\medskip
\noindent\textbf{Case 2: $p_j\leq a_j<1$.} Here $p_j^2\leq \mu_{k_j}(\F_j)< p_j$, the second branch of \cref{thm:shadow-main} gives
$$\mu_{qn}(\partial_{k_j\to qn}\F_j)-q\geq -q\left(1-a_j\right)^\beta.$$
Combining this with~\eqref{ineq:uniform-critical-sum-constraint}, we obtain $\sum_{i<j}p_i(a_i-1)\leq q(1-a_j)^\beta.$ Part (ii) of \cref{thm:product-inequality} now yields $\prod_{i=1}^r a_i<1$.

\medskip
\noindent\textbf{Case 3: $0<a_j<p_j$.} Now
$\mu_{k_j}(\F_j)=p_ja_j<p_j^2$, so the first branch of
\cref{thm:shadow-main} gives
$$\mu_{qn}(\partial_{k_j\to qn}\F_j)-q\geq -q+(p_ja_j)^\alpha=q(a_j^{\alpha}-1).$$
Then~\eqref{ineq:uniform-critical-sum-constraint} implies $\sum_{i<j}p_i(a_i-1)\leq q(1-a_j^{\alpha})$. Part (i) of \cref{thm:product-inequality} again gives $\prod_{i=1}^ra_i<1$.

Suppose that equality holds. Cases~2 and~3 are excluded by their strict bounds, so Case~1 applies. There $a_i=1$ for $i\leq j$ and $a_i\leq1$ for $i>j$. Since $\prod_i a_i=1$, it follows that $a_i=1$ for every $i\in[r]$, or equivalently $\mu_{k_i}(\F_i)=\frac{k_i}{n}$. We use this conclusion in the equality classification below.
\end{proof}

The following lemma shows that the full $1$-stars arising in the
equality argument have a common center.

\begin{lemma}\label{lem:common-star-center}
Let $r\geq2$ and $1\leq k_i<n$ for every $i\in[r]$, and suppose that $\sum_{i=1}^r(n-k_i)\geq n$. For $c_i\in[n]$, let
$$\mathcal S_i=\left\{A\in\binom{[n]}{k_i}:c_i\in A\right\}.$$
If $\mathcal S_1,\ldots,\mathcal S_r$ are $r$-cross-intersecting, then $c_1=\cdots=c_r$.
\end{lemma}

\begin{proof}
Suppose, to the contrary, that the centers are not all equal. Then $c_s\neq c_t$ for some $s,t\in[r]$. Assign each
$x\in[n]$ to an index $i$ satisfying $x\neq c_i$. Such an index exists for every $x$: if $x\neq c_s$, choose $i=s$; otherwise $x=c_s\neq c_t$, so choose $i=t$. Thus we obtain an ordered partition $[n]=Q_1\mathbin{\dot\cup}\cdots\mathbin{\dot\cup}Q_r$ with $c_i\notin Q_i$ for every $i$, where empty parts are allowed.

We may further arrange that $|Q_i|\leq n-k_i$ for all $i\in[r]$. Indeed, if $|Q_i|>n-k_i$ for some $i$, then since $\sum_{j=1}^r|Q_j|=n\leq \sum_{j=1}^r(n-k_j)$,
there is some $j$ with $|Q_j|<n-k_j$. As $n-k_i\geq 1$, the set $Q_i$ has at least two elements, so $Q_i$ contains an element $x\neq c_j$. Moving such an $x$ from $Q_i$ to $Q_j$ preserves the condition $c_t\notin Q_t$ and decreases
$$\sum_{t=1}^r\max\{0,|Q_t|-(n-k_t)\}$$
by one. Repeating this step gives a partition satisfying
$|Q_i|\leq n-k_i$ for every $i$.

For each $i$, extend $Q_i$ to a set $P_i\subseteq[n]\setminus\{c_i\}$ of size $n-k_i$. Then
$[n]\setminus P_i\in \mathcal S_i$ for every $i$, whereas
$$\bigcap_{i=1}^r([n]\setminus P_i)=[n]\setminus\bigcup_{i=1}^rP_i=\varnothing.$$
This contradicts $r$-cross-intersection. Therefore, $c_1=\cdots=c_r$.
\end{proof}

\begin{proof}[Proof of equality in~\cref{thm:uniform-main}]
\noindent\textbf{Case 1: Some $k_i=0$.}
The right-hand side of~\eqref{ineq:uniform-main} is $0$. If $\F_i$ is empty, then the left-hand side is also $0$. Otherwise
$\F_i=\{\varnothing\}$, and cross-intersection forces another
family to be empty. Thus equality holds for all $r$-cross-intersecting families in
this case.

Assume from now on that each $k_i>0$ and that equality holds. The preceding proof gives $\mu_{k_i}(\F_i)=\frac{k_i}{n}$ for every $i\in[r]$. Since $k_i\leq\frac{r-1}{r}n$, we have $\sum_{i=1}^r(n-k_i)\geq n$.

\medskip
\noindent\textbf{Case 2:
$\sum_{i=1}^r(n-k_i)=n$.}
Since $k_i\leq n(r-1)/r$, we have $k_i=\frac{r-1}{r}n$ for every $i\in[r]$. When $r=2$, two middle-level sets are disjoint exactly when they are complements. Hence
$$\F_2\subseteq\left\{B\in\binom{[n]}{n/2}:[n]\setminus B\notin\F_1\right\}.$$
Both sides have cardinality $\frac12\binom{n}{n/2}$, so equality holds in this inclusion. This gives precisely the exceptional pairs in the theorem. Suppose that $r\geq3$. By \cref{lem:uniform-partition-equality}, either the families $\F_1,\dots,\F_r$ are the levels of a common $1$-star, or one family is empty and all the others are full. The latter alternative is impossible because every family has density $(r-1)/r$. Hence the families $\F_1,\dots,\F_r$ are the levels of a common $1$-star.

\medskip
\noindent\textbf{Case 3: $\sum_{i=1}^r(n-k_i)>n$.}
Fix $i\in[r]$. We show that $\F_i$ is a $1$-star. If $k_i=n-1$, then $|\F_i|=n-1$, and any family of $n-1$ members of $\binom{[n]}{n-1}$ is a $1$-star. If $k_i\leq n-2$, then the strict inequality gives
$$(k_i+1)+\sum_{t\neq i}k_t\leq (r-1)n\leq(n-1)+\sum_{t\neq i}n.$$
Hence there are integers $\ell_1,\ldots,\ell_r$ satisfying $k_i<\ell_i<n$, $k_t\leq\ell_t\leq n$ for all $t\neq i$, and $\sum_{t=1}^r\ell_t=(r-1)n$.
The families $\partial_{k_t\to\ell_t}\F_t$ remain $r$-cross-intersecting. If $\ell_t<n$, then \eqref{eq:uniform-power-comparison} and
$\mu_{k_t}(\F_t)=k_t/n$ give
$$
\mu_{\ell_t}(\partial_{k_t\to\ell_t}\F_t)
\geq
\left(\frac{k_t}{n}\right)^{
\frac{\log(\ell_t/n)}{\log(k_t/n)}}
=\frac{\ell_t}{n}.
$$
If $\ell_t=n$, then $\F_t\neq\varnothing$ and
$\partial_{k_t\to n}\F_t=\{[n]\}$, so the same conclusion holds.
The critical-sum inequality therefore gives
$$
r-1=\sum_{t=1}^r\frac{\ell_t}{n}\leq
\sum_{t=1}^r
\mu_{\ell_t}(\partial_{k_t\to\ell_t}\F_t)
\leq r-1.
$$
Therefore,
$$
\mu_{\ell_i}(\partial_{k_i\to\ell_i}\F_i)
=\frac{\ell_i}{n}.
$$
Together with $\mu_{k_i}(\F_i)=k_i/n$ and
$k_i<\ell_i<n$, \cref{cor:shadow-at-star} now shows that $\F_i$ is a $1$-star. Since $i$ was arbitrary, all the $\F_i$ are $1$-stars, and \cref{lem:common-star-center} forces their centers to coincide.

Conversely, the corresponding levels of a common $1$-star are
$r$-cross-intersecting and attain equality. The exceptional middle-level pairs are also cross-intersecting and have density $\frac{1}{2}$ in both coordinates. Together with Case~1, this proves all the equality statements in~\cref{thm:uniform-main}.
\end{proof}

The biased theorem uses the same normalization and critical-sum
argument.

\begin{proof}[Proof of \cref{thm:biased-main}]
Suppose first that $p_i=0$ for some $i$. If
$\varnothing\notin\A_i$, then $\mu_0(\A_i)=0$. If
$\varnothing\in\A_i$, the cross-intersection condition forces at
least one of the other families to be empty. Hence the product on the
left is zero in either case, and so equality holds for every
collection satisfying the hypotheses when some $p_i=0$.

Assume from now on that $p_i>0$ for every $i$. If some $\A_i$ is empty, the desired inequality is immediate. For each $i$,
let
$$
\A_i^\uparrow=
\{B\subseteq[n]:A\subseteq B\text{ for some }A\in\A_i\}.
$$
The upward closures remain $r$-cross-intersecting and satisfy
$\mu_{p_i}(\A_i)\leq\mu_{p_i}(\A_i^\uparrow)$. It is therefore
enough to prove the inequality for the upward closures.

Set $a_i=\frac{\mu_{p_i}(\A_i^\uparrow)}{p_i}$, and relabel the families so that $a_1\geq\cdots\geq a_r>0$. Let $j$ be the first index satisfying
$\sum_{i\leq j}(1-p_i)\geq1$, and put $q=\sum_{i<j}(1-p_i)$. Such an index exists because $1-p_i\geq \frac{1}{r}$. Its minimality gives $p_j\leq q<1$, and
$$
\sum_{i<j}p_i+q+(r-j)=r-1.
$$
Every nonempty increasing family contains $[n]$, so $\mu_1(\A_i^\uparrow)=1$. Applying \cref{lem:biased-critical-sum} with
biases $p_1,\ldots,p_{j-1},q,1,\ldots,1$ gives
\begin{equation}\label{eq:biased-main-constraint}
\sum_{i<j}p_i(a_i-1)+\mu_q(\A_j^\uparrow)-q\leq0.
\end{equation}
Applying \cref{thm:biased-shadow-comparison} to
$\A_j^\uparrow$ and treating the same three ranges for $a_j$ as in the
uniform proof, we combine \eqref{eq:biased-main-constraint} with
\cref{thm:product-inequality} to obtain $\prod_{i=1}^r a_i\leq1$. The same argument shows that equality for nonempty increasing families
forces
\begin{equation}\label{eq:biased-equality-density}
\mu_{p_i}(\A_i^\uparrow)=p_i,
\end{equation}
for every $i\in[r]$.

We now classify equality for the original families. If equality
holds, then
$$
\prod_i\mu_{p_i}(\A_i)
\leq\prod_i\mu_{p_i}(\A_i^\uparrow)
\leq\prod_i p_i
=\prod_i\mu_{p_i}(\A_i).
$$
Every factor is positive, and each factor in the middle product is at
least the corresponding factor in the first. Hence
$\mu_{p_i}(\A_i^\uparrow)=\mu_{p_i}(\A_i)$ for every $i$. Thus
$\A_i^\uparrow=\A_i$ and $\mu_{p_i}(\A_i)=p_i$. By
\eqref{eq:biased-power-comparison},
$$
\mu_{(r-1)/r}(\A_i)
\geq\mu_{p_i}(\A_i)^{
\frac{\log((r-1)/r)}{\log p_i}}
=\frac{r-1}{r}
$$
for every $i\in[r]$. Applying \cref{lem:biased-critical-sum} with
all biases equal to $(r-1)/r$ gives $\sum_{i=1}^r\mu_{(r-1)/r}(\A_i)\leq r-1$. Consequently,
\begin{equation}\label{eq:critical-biased-equality}
\mu_{(r-1)/r}(\A_i)=\frac{r-1}{r},
\end{equation}
for every $i\in[r]$.

Suppose first that $r\geq3$. By
\eqref{eq:critical-biased-equality} and \cref{lem:biased-equality},
there is $x\in[n]$ such that $\A_i=\left\{A\subseteq [n]:x\in A\right\}$ for every $i\in[r]$.

It remains to consider $r=2$. Cross-intersection gives
\begin{equation}\label{ineq:biased-complement-inclusion}
    \A_2\subseteq\{B\subseteq[n]:[n]\setminus B\notin\A_1\}.
\end{equation}
If $p_1=p_2=\frac{1}{2}$, then both sides of \eqref{ineq:biased-complement-inclusion} have $\mu_{1/2}$-measure $\frac{1}{2}$, so the inclusion is an equality. If $(p_1,p_2)\neq(\frac{1}{2},\frac{1}{2})$, then, without loss of generality, we may assume that $p_1<\frac{1}{2}$. Equations~\eqref{eq:biased-equality-density} and~\eqref{eq:critical-biased-equality} give $\mu_{p_1}(\A_1)=p_1$ and $\mu_{1/2}(\A_1)=\frac12$. Thus equality holds in \eqref{eq:biased-power-comparison} for
$p=p_1$ and $q=\frac{1}{2}$. The equality statement in~\cref{thm:biased-shadow-comparison} shows that $\A_1$ is a full $1$-star. Then~\eqref{ineq:biased-complement-inclusion} and $\mu_{p_2}(\A_2)=p_2$ force $\A_2=\{B\subseteq[n]:[n]\setminus B\notin\A_1\}=\A_1$.

Conversely, a common full $1$-star attains equality. In the exceptional
case, let $\A$ be increasing and put
$$
\A^*=\{B\subseteq[n]:[n]\setminus B\notin\A\}.
$$
The families $\A$ and $\A^*$ are
cross-intersecting, and when $\mu_{1/2}(\A)=\frac{1}{2}$ both have measure $\frac{1}{2}$. This proves the equality statement.
\end{proof}

\section{The analytic product inequality}\label{sec:product-inequality}
We now prove \cref{thm:product-inequality}. Set $\alpha=\frac{\log q}{\log p_j}$ and $\beta=\frac{\log(1-q)}{\log(1-p_j)}$, and define
$\Delta:[0,1]\to[0,1]$ by
\begin{equation}\label{eq:delta-definition}
\Delta(y)=
\begin{cases}
1-y^\alpha,&0\leq y\leq p_j,\\[1mm]
(1-y)^\beta,& p_j\leq y\leq1.
\end{cases}
\end{equation}
In either case, the hypothesis becomes
$$\sum_{i<j}p_i(a_i-1)\leq q\Delta(a_j).$$
We shall use the following properties of $\Delta$. 

\begin{lemma}\label{lem:delta-properties}
\begin{enumerate}
    \item[(i)] $0<\alpha\leq1\leq\beta$ and
    $q^2\leq \frac{r-1}{r-1+\alpha}$.
    \item[(ii)] The function $\Delta$ is continuous and nonincreasing,
    with $\Delta(0)=1$ and $\Delta(1)=0$.
    \item[(iii)] For every $0\leq y\leq1$,
    $$
    y^{r-1}\bigl(1-q+q\Delta(y)\bigr)\leq1-q,
    $$
    with equality if and only if $y=1$.
\end{enumerate}
\end{lemma}

\begin{proof}
The minimality of $j$ gives $q<1$ and
$q+1-p_j\geq1$, hence $p_j\leq q<1$. Therefore, $0<\alpha\leq1\leq\beta$. Moreover,
$$
\log\left(1+\frac{\alpha}{r-1}\right)
\leq\frac{\alpha}{r-1}
\leq\frac{2\alpha}{r}
\leq2\alpha\log\left(1+\frac1{r-1}\right).
$$
It follows that
$$
q^2=p_j^{2\alpha}
\leq\left(\frac{r-1}{r}\right)^{2\alpha}
\leq\frac{r-1}{r-1+\alpha}.
$$
This proves part (i).

At $y=p_j$, the two lines of \eqref{eq:delta-definition} agree,
because $p_j^\alpha=q$ and $(1-p_j)^\beta=1-q$.
Each line is nonincreasing on its interval, proving part (ii).

It remains to prove (iii). First suppose that $0\leq y\leq p_j$.
Then
$$
1-q+q\Delta(y)=1-qy^\alpha.
$$
The derivative of $y^{r-1}(1-qy^\alpha)$ is $y^{r-2}\bigl((r-1)-(r-1+\alpha)qy^\alpha\bigr)$.
Since
$$
qy^\alpha\leq qp_j^\alpha=q^2
\leq\frac{r-1}{r-1+\alpha},
$$
this derivative is nonnegative. Therefore
\begin{align*}
y^{r-1}\bigl(1-q+q\Delta(y)\bigr)\leq p_j^{r-1}(1-q^2)=(1+q)(1-q)p_j^{r-1}\leq2(1-q)\left(\frac{r-1}{r}\right)^{r-1}\leq1-q,
\end{align*}
where the last inequality follows from
$(1+1/(r-1))^{r-1}\geq2$. Since $q<1$, the inequality on this
interval is strict.

Now suppose that $p_j\leq y\leq1$. Put $z=\frac{1-y}{1-p_j}$ and $q_j=1-p_j$.
Then $0\leq z\leq1$, $q_j\geq1/r$,
$1-q=q_j^\beta$, and
$$
y^{r-1}\bigl(1-q+q\Delta(y)\bigr)
=q_j^\beta(1-q_jz)^{r-1}\bigl(1+qz^\beta\bigr).
$$
For $0<z\leq1$, convexity of $u\mapsto u^\beta$ gives
$$
\frac{z^\beta-(q_jz)^\beta}{z-q_jz}
\leq
\frac{1-(q_jz)^\beta}{1-q_jz}
\leq\frac{1}{1-q_jz},
$$
and hence
$qz^\beta=z^\beta-(q_jz)^\beta
\leq\frac{z(1-q_j)}{1-q_jz}.$
On the other hand,
$$
(1-q_jz)^{-(r-1)}-1
\geq\frac{(r-1)q_jz}{1-q_jz}
\geq\frac{z(1-q_j)}{1-q_jz},
$$
where the first inequality is Bernoulli's inequality applied to
$$
(1-q_jz)^{-(r-1)}
=\left(1+\frac{q_jz}{1-q_jz}\right)^{r-1},
$$
and the second uses $(r-1)q_j\geq1-q_j$. Thus
$1+qz^\beta\leq(1-q_jz)^{-(r-1)}$.  If $z>0$, at least one
inequality in this chain is strict.  Indeed, Bernoulli's inequality
is strict when $r>2$, while for $r=2$ the second inequality is strict
unless $q_j=1/2$.  In the remaining case $r=2$ and $q_j=1/2$, the
inequality
$$
\frac{1-(q_jz)^\beta}{1-q_jz}
<\frac{1}{1-q_jz}
$$
is strict.  Hence equality is impossible when $z>0$, or equivalently
when $y<1$.  At $z=0$, that is, $y=1$, equality holds. This proves part (iii).
\end{proof}

\begin{proof}[Proof of \cref{thm:product-inequality}]
Since $p_1\leq \frac{r-1}{r}$, we have $j\geq2$. Put $p=1-q+\frac{j-2}{r}$. Then
\begin{equation}\label{eq:weight-sum}
\sum_{i=1}^{j-1}p_i
=p+(j-2)\frac{r-1}{r}.
\end{equation}
Moreover, for every $i<j$,
$p_i-p=\sum_{\substack{t<j\\t\ne i}}\left(\frac{r-1}{r}-p_t\right)\geq 0$.
Using \eqref{eq:weight-sum} and $a_1\geq a_i$, we obtain
\begin{align*}
&\sum_{i=1}^{j-1}p_i(a_i-1)
-
\left[
 p(a_1-1)
 +\frac{r-1}{r}\sum_{i=2}^{j-1}(a_i-1)
\right]=\sum_{i=2}^{j-1}\left(\frac{r-1}{r}-p_i\right)(a_1-a_i)\geq 0.
\end{align*}
Let $T(y)=j-1-q+q\Delta(y)$. By the theorem hypothesis
$\sum_{i<j}p_i(a_i-1)\leq q\Delta(a_j)$. We therefore obtain
\begin{equation}\label{eq:compressed-budget}
p a_1+\frac{r-1}{r}\sum_{i=2}^{j-1}a_i
\leq p+(j-2)\frac{r-1}{r}+q\Delta(a_j)=T(a_j).
\end{equation}

Because $a_i\leq a_j$ for $i\geq j$,
\begin{equation}\label{eq:tail-compression}
\prod_{i=1}^{r}a_i
\leq a_j^{r-j+1}\prod_{i=1}^{j-1}a_i.
\end{equation}
If $j=2$, then \eqref{eq:compressed-budget} and
\eqref{eq:tail-compression} give
$$
\prod_{i=1}^r a_i
\leq
a_2^{r-1}a_1
\leq
\frac{a_2^{r-1}T(a_2)}{p}
<1,
$$
where the last inequality follows from~\cref{lem:delta-properties}(iii),
since $a_2<1$ and in this case $p=1-q$.  Hence we may assume that
$j\geq3$; in particular, $r\geq3$.
Let $z=\frac{1}{j-2}\sum_{i=2}^{j-1}a_i$.
Then $z\geq a_j$, and AM--GM together with
\eqref{eq:compressed-budget} gives
\begin{equation}\label{eq:one-variable-product}
\prod_{i=1}^{r}a_i
\leq
a_j^{r-j+1}\frac{T(a_j)-\frac{r-1}{r}(j-2)z}{p}
z^{j-2}.
\end{equation}
For fixed $a_j$, differentiation gives
\begin{align*}
&\frac{d}{dz}
\left[\left(T(a_j)-\frac{r-1}{r}(j-2)z\right)z^{j-2}
\right]=(j-2)z^{j-3}
\left(
T(a_j)-(j-1)\frac{r-1}{r}z
\right).
\end{align*}
By \cref{lem:delta-properties}(ii), $T$ is nonincreasing. Moreover,
$T(0)=j-1$, and
$$
q=\sum_{i<j}(1-p_i)\geq\frac{j-1}{r}.
$$
Hence
$T(1)-(j-1)(r-1)/r=(j-1)/r-q\leq0$.  Since
$T(y)-(j-1)(r-1)y/r$ is strictly decreasing, there is a unique
$y_0\in(0,1]$ such that
\begin{equation}\label{eq:y-zero}
T(y_0)=(j-1)\frac{r-1}{r}y_0.
\end{equation}
The derivative in \eqref{eq:one-variable-product} therefore changes
sign at $y_0$, leading to the following two cases.

\textbf{Case 1: $a_j\geq y_0$.} Then $T(a_j)-(j-1)\frac{r-1}{r}a_j\leq 0$, so the right-hand side of
\eqref{eq:one-variable-product} is maximized at $z=a_j$. Hence
$$\prod_{i=1}^{r}a_i\leq\frac{a_j^{r-1}}{p}\left(T(a_j)-(j-2)\frac{r-1}{r}a_j
\right).$$
Since $$\frac{d}{dy}
\left[y^{r-1}\left(1-\frac{r-1}{r}y\right)\right]=(r-1)y^{r-2}(1-y)\geq 0,$$
we have $y^{r-1}\left(1-\frac{r-1}{r}y\right)\leq \frac1r$.  Also,
$$
T(y)-(j-2)\frac{r-1}{r}y
=
1-q+q\Delta(y)
+
(j-2)\left(1-\frac{r-1}{r}y\right).
$$
Combining these facts with
\cref{lem:delta-properties}(iii), we get, for all $0\leq y\leq1$,
\begin{equation}\label{eq:key-bound}
y^{r-1}\left(T(y)-(j-2)\frac{r-1}{r}y\right)\leq p.
\end{equation}
Both estimates used here are strict for $y<1$.  Hence
\eqref{eq:key-bound} is strict for $y<1$, and since $a_j<1$ we obtain
$\prod_{i=1}^{r}a_i<1$.

\textbf{Case 2: $a_j\leq y_0$.} Then
$T(a_j)- (j-1)\frac{r-1}{r}a_j\geq 0$ and the maximizing value is $z=\frac{T(a_j)}{(j-1)\frac{r-1}{r}}$. Hence
$$\prod_{i=1}^{r}a_i\leq a_j^{r-j+1}\frac{T(a_j)^{j-1}}{p(j-1)^{j-1}(\frac{r-1}{r})^{j-2}}.$$

\begin{claim}\label{clm:g-monotone} If $j\geq3$, then the function
$G(y)=y^{r-j+1}T(y)^{j-1}$ is strictly increasing on $[0,1]$.
\end{claim}
The proof of \cref{clm:g-monotone} is given after the proof of the
theorem.

If $a_j<y_0$, the strict increase of $G$ makes the first inequality
below strict.  If $a_j=y_0$, then $y_0<1$ and the final inequality is
strict by \eqref{eq:key-bound}.  Thus, using \eqref{eq:y-zero},
\begin{align*}
\prod_{i=1}^{r}a_i\leq
\frac{G(y_0)}
{p(j-1)^{j-1}\left(\frac{r-1}{r}\right)^{j-2}}
=\frac{\frac{r-1}{r}y_0^r}{p}\leq1,
\end{align*}
and at least one of the two inequalities is strict. This completes
the proof of the theorem.
\end{proof}

\begin{proof}[Proof of \cref{clm:g-monotone}]
For $0<y\leq p_j$,
\eqref{eq:delta-definition} gives $T(y)=j-1-qy^\alpha$.
Therefore
\begin{align*}
G'(y)=y^{r-j}T(y)^{j-2}
\Bigl(
(r-j+1)(j-1)-qy^\alpha\bigl(r-j+1+(j-1)\alpha\bigr)
\Bigr).
\end{align*}
By $q=p_j^{\alpha}$ and \cref{lem:delta-properties}(i), $qy^\alpha\leq q^2\leq \frac{r-1}{r-1+\alpha}$.
Moreover, since $j\geq 3$ and $0< \alpha\leq 1$,
\begin{align*}
\frac{(r-j+1)(j-1)}{r-j+1+(j-1)\alpha}
-\frac{r-1}{r-1+\alpha}=
\frac{(j-2)\bigl((r-j+1)(r-1)-(j-1)\alpha\bigr)}{\bigl(r-j+1+(j-1)\alpha\bigr)(r-1+\alpha)}\geq 0.
\end{align*}
Hence $G'(y)\geq0$ on $(0,p_j)$.  In fact $G'(y)>0$ there, since
$qy^\alpha<q p_j^\alpha=q^2$ for $y<p_j$.

Now let $p_j<y<1$. Then $T(y)=j-1-q+q(1-y)^\beta$.
The sign of $G'(y)$ is the sign of
\begin{equation}\label{eq:claim-derivative-second-branch}
(r-j+1)T(y)
-(j-1)q\beta y(1-y)^{\beta-1}.
\end{equation}
Since $\beta\geq1$,
$$
1-(1-y)^\beta
=\int_{1-y}^1\beta t^{\beta-1}\,dt
\geq\beta y(1-y)^{\beta-1}.
$$
Also, $q\bigl(1-(1-y)^\beta\bigr)=j-1-T(y)$.
Thus \eqref{eq:claim-derivative-second-branch} is at least
\begin{align*}
(r-j+1)T(y)-(j-1)\bigl(j-1-T(y)\bigr)=rT(y)-(j-1)^2\geq (r-j+1)(j-1)-rq.
\end{align*}
If $3\leq j\leq r-1$, then $j-1\geq2$ and $r-j+1\geq2$, so
$(r-j+1)(j-1)\geq r$.  Since $q<1$, the last expression is strictly
positive.

It remains to consider the endpoint case $j=r$. Put $q_j=1-p_j$ and $t=1-y$.
Then $q_j\geq1/r$ and $1-q=q_j^\beta$. In this case the
nonnegativity of \eqref{eq:claim-derivative-second-branch} is equivalent
to
\begin{equation}\label{eq:endpoint-inequality}
r-2+q_j^\beta+(1-q_j^\beta)t^\beta
\geq
(r-1)(1-q_j^\beta)\beta t^{\beta-1}(1-t).
\end{equation}
For fixed $\beta$, the difference between the two sides is increasing as a function of $q_j^\beta$, because the coefficient of $q_j^\beta$ is $1-t^\beta+(r-1)\beta t^{\beta-1}(1-t)\geq 0$.
Hence it is enough to prove \eqref{eq:endpoint-inequality} for $q_j=1/r$.
Equivalently, it is enough to show
\begin{equation*}
(r-1)\beta t^{\beta-1}(1-t)-t^\beta
\leq
\frac{r-2+r^{-\beta}}{1-r^{-\beta}}.
\end{equation*}

If $\beta=1$, the left-hand side is $r-1-rt\leq r-1$,
while the right-hand side is $r-1$.  We may therefore assume that
$\beta>1$.

We consider the function $F(t)=(r-1)\beta t^{\beta-1}(1-t)-t^\beta$. Then
$$F'(t)=\beta t^{\beta-2}\left((r-1)(\beta-1)-\bigl(r+(r-1)(\beta-1)\bigr)t\right).$$
Hence $F(t)$ is increasing on $[0,\tau]$ and decreasing on $[\tau,1]$, where $\tau=\frac{(r-1)(\beta-1)}{r+(r-1)(\beta-1)}$. Therefore, $F(t)$ attains its maximum at $t=\tau$, and $F(\tau)=(r-1)\tau^{\beta-1}$. It therefore suffices to prove $\tau^{\beta-1}\leq \frac{r-2+r^{-\beta}}{(r-1)(1-r^{-\beta})}$.

Suppose first that $1<\beta\leq 2$. Concavity gives
$1-\tau^{\beta-1}\geq (\beta-1)(1-\tau)=\frac{(\beta-1) r}{r+(r-1)(\beta-1)}$, so the desired inequality follows from
$$\frac{(\beta-1) r}{r+(r-1)(\beta-1)}\geq 1-\frac{r-2+r^{-\beta}}{(r-1)(1-r^{-\beta})}=\frac{1-r^{1-\beta}}{(r-1)(1-r^{-\beta})}.$$
After clearing denominators, this is
$r(1-r^{1-\beta})\leq(\beta-1)(r-1)^2$, which follows from $1-r^{1-\beta}\leq(\beta-1)\log r$ and
$r\log r\leq(r-1)^2$ for $r\geq 3$.

If $\beta\geq 2$, the function
$\left(\frac{(r-1)(\beta-1)}{r+(r-1)(\beta-1)}\right)^{\beta-1}$ is
nonincreasing in $\beta$, because its logarithmic derivative is
$\log(1-y)+y\leq 0$, where
$y=\frac{r}{r+(r-1)(\beta-1)}$. Hence
$$\tau^{\beta-1}\leq\frac{r-1}{2r-1}\leq\frac{r-2}{r-1}\leq\frac{r-2+r^{-\beta}}{(r-1)(1-r^{-\beta})}.$$
It remains only to verify strict increase when $j=r$.  If $\beta=1$,
then $q=p_j=(r-1)/r$, and the expression in
\eqref{eq:claim-derivative-second-branch} equals
$(r-1)(1-y)>0$.  If $\beta>1$, the proof above shows that this
expression is nonnegative and analytic on $(p_j,1)$.  Its limit as
$y\uparrow1$ is $r-1-q>0$, so it is not identically zero on any
subinterval.  It follows that $G$ is strictly increasing on
$(p_j,1)$.  Together with the strict increase on $(0,p_j)$ and
continuity at $p_j$, this proves the claim.
\end{proof}

\section{Concluding remarks}\label{sec:concluding}

\paragraph{$r$-cross $t$-intersection.}
Our results settle the case $t=1$ of two broader problems posed by
Frankl and Tokushige~\cite[Problems~12.10 and~12.11]{FT2018book}. For a
positive integer $t$, families
$\A_1,\ldots,\A_r\subseteq2^{[n]}$ are called $r$-cross $t$-intersecting if
$|A_1\cap\cdots\cap A_r|\geq t$
whenever $A_i\in\A_i$ for every $i\in[r]$.

For $t\geq2$, the same coupling gives an exact critical-sum
inequality. If $\sum_i\ell_i=(r-1)n+t-1$, choose an ordered
decomposition
$$
[n]=C\mathbin{\dot\cup}P_1\mathbin{\dot\cup}\cdots
\mathbin{\dot\cup}P_r,
\qquad |C|=t-1,\quad |P_i|=n-\ell_i,
$$
and set $X_i=[n]\setminus P_i$. Then
$|X_1\cap\cdots\cap X_r|=t-1$, so every collection of $r$-cross
$t$-intersecting families $\F_i\subseteq\binom{[n]}{\ell_i}$ satisfies
$$
\sum_{i=1}^r\mu_{\ell_i}(\F_i)\leq r-1.
$$
A corresponding product theorem would require an analogue of
\cref{thm:shadow-main,cor:shadow-at-star}. A common $t$-star has density
$\binom{k}{t}/\binom{n}{t}$, whereas in other parameter ranges an
Ahlswede--Khachatrian family may be extremal. Any such comparison must
therefore incorporate the relevant extremal families and their
dependence on the parameters.
\medskip

\noindent\textbf{Acknowledgments.} The authors thank Mengyu Cao,
Ting-wei Chao, Xingtong Guo, Guowei Sun, and Haixiang Zhang for
valuable discussions during the fourth ECOPRO Student Research
Program, held at the Institute for Basic Science in the summer of
2026. AI tools were used during the exploratory stage of the project.
All mathematical
arguments and proofs in the manuscript were developed and rigorously verified by
the authors.

\bibliographystyle{abbrv}
\bibliography{reference1}

\end{document}